\pdfoutput=1
\RequirePackage{ifpdf}
\ifpdf % We~are running pdfTeX in pdf mode
\documentclass[pdftex]{sigma}
\else
\documentclass{sigma}
\fi

\numberwithin{equation}{section}

\newtheorem{Theorem}{Theorem}[section]
\newtheorem{Corollary}[Theorem]{Corollary}
\newtheorem{Lemma}[Theorem]{Lemma}

\theoremstyle{definition}
\newtheorem{Definition}[Theorem]{Definition}

\newtheorem{Example}[Theorem]{Example}
\newtheorem{Remark}[Theorem]{Remark}

\DeclareSymbolFont{AMSb}{U}{msb}{m}{n}
\DeclareMathSymbol{\N}{\mathalpha}{AMSb}{"4E}
\DeclareMathSymbol{\R}{\mathalpha}{AMSb}{"52}
\DeclareMathSymbol{\Z}{\mathalpha}{AMSb}{"5A}
\DeclareMathSymbol{\D}{\mathalpha}{AMSb}{"44}
\DeclareMathSymbol{\s}{\mathalpha}{AMSb}{"53}

\DeclareMathOperator{\OptGeo}{OptGeo}

\newcommand{\Y}{{Y}}

\newcommand{\kkapp}{\chi}
\newcommand{\RCRD}{\mathrm{RCD}}
\newcommand{\CDD}{\mathrm{CD}}
\newcommand{\MCP}{\mathrm{MCP}}
\DeclareMathOperator{\inrad}{InRad}

\DeclareMathOperator{\sgn}{sign}

\DeclareMathOperator{\vol}{vol}
\DeclareMathOperator{\Ch}{Ch}
\DeclareMathOperator{\lip}{Lip}

\DeclareMathOperator{\supp}{{spt}}

\DeclareMathOperator{\m}{m}
\DeclareMathOperator{\dm}{dm}
\DeclareMathOperator{\ric}{ric}
\DeclareMathOperator{\diam}{diam}

\newcommand{\T}{\mathcal{T}}

\DeclareMathOperator{\sign}{sign}

\DeclareMathOperator{\id}{id}

\begin{document}
\allowdisplaybreaks

\newcommand{\arXivNumber}{2005.07435}

\renewcommand{\thefootnote}{}

\renewcommand{\PaperNumber}{131}

\FirstPageHeading

\ShortArticleName{Inscribed Radius Bounds for Lower Ricci Bounded Metric Measure Spaces}

\ArticleName{Inscribed Radius Bounds for Lower Ricci Bounded\\ Metric Measure Spaces with Mean Convex Boundary\footnote{This paper is a~contribution to the Special Issue on Scalar and Ricci Curvature in honor of Misha Gromov on his 75th Birthday. The full collection is available at \href{https://www.emis.de/journals/SIGMA/Gromov.html}{https://www.emis.de/journals/SIGMA/Gromov.html}}}

\Author{Annegret BURTSCHER~$^\dag$, Christian KETTERER~$^\ddag$, Robert J.~MCCANN~$^\ddag$\newline and Eric WOOLGAR~$^\S$}

\AuthorNameForHeading{A.~Burtscher, C.~Ketterer, R.J.~McCann and E.~Woolgar}

\Address{$^\dag$~Department of Mathematics, IMAPP, Radboud University, \\
\hphantom{$^\dag$}~PO Box 9010, Postvak 59, 6500 GL Nijmegen, The Netherlands}
\EmailD{\href{mailto:burtscher@math.ru.nl}{burtscher@math.ru.nl}}

\Address{$^\ddag$~Department of Mathematics, University of Toronto,\\
\hphantom{$^\ddag$}~40 St George St, Toronto Ontario, Canada M5S 2E4}
\EmailD{\href{mailto:ckettere@math.toronot.edu}{ckettere@math.toronto.edu}, \href{mailto:mccann@math.toronto.edu}{mccann@math.toronto.edu}}

\Address{$^\S$~Department of Mathematical and Statistical Sciences and Theoretical Physics Institute,\\
\hphantom{$^\S$}~University of Alberta, Edmonton AB, Canada T6G 2G1}
\EmailD{\href{mailto:ewoolgar@ualberta.ca}{ewoolgar@ualberta.ca}}

\ArticleDates{Received June 03, 2020, in final form November 21, 2020; Published online December 10, 2020}

\Abstract{Consider an essentially nonbranching metric measure space with the measure contraction property of Ohta and Sturm, or with a Ricci curvature lower bound in the sense of Lott, Sturm and Villani. We prove a sharp upper bound on the inscribed radius of any subset whose boundary has a suitably signed lower bound on its generalized mean curvature. This provides a nonsmooth analog to a result of Kasue (1983) and Li (2014). We prove a~stability statement concerning such bounds and~-- in the Riemannian curvature-dimension $(\RCRD)$ setting~-- characterize the cases of equality.}


\Classification{51K10; 53C21; 30L99; 83C75}

\renewcommand{\thefootnote}{\arabic{footnote}}
\setcounter{footnote}{0}

\section{Introduction}
Kasue proved a sharp estimate for the inscribed radius (or inradius, denoted $\inrad$) of a smooth, $n$-dimensional Riemannian manifold $M$ with nonnegative Ricci curvature and smooth boundary $\partial M$ whose mean curvature is bounded from below by $n-1$. More precisely, he concluded $\inrad_M\leq 1$ \cite{Kasue83}.
This result was also rediscovered by Li \cite{martinli} and extended
to weighted Riemannian manifolds with Bakry--\'Emery curvature bounds by Li--Wei~\cite{liwei_bakry_emery, liwei_rigidity} and Sakurai~\cite{sakurai}.
Their result can be seen either as a manifold-with-boundary analog of Bonnet and Myers' diameter bound,
or as a Riemannian analog of the Hawking singularity theorem from general rela\-tivity~\cite{Hawking66} (for the precise statement see \cite[Theorem~6.49]{minguzzi}). There has been considerable interest in generalizing Hawking's result
 to a nonsmooth setting \cite{Graf19+, KunzingerSteinbauerStojkovicVickers15, LuMinguzziOhta19+}.
 Motivated in part by this goal, we give a generalization of Kasue's result which is interesting in itself and can serve as a model for the Lorentzian case. Independently and simultaneously, Cavalletti and Mondino have proposed a synthetic new framework for Lorentzian geometry (also under investigation by one of us independently~\cite{McCann18+}) in which they establish an analog of the Hawking result~\cite{cm_new}.

In this note we generalize Kasue and Li's estimate to subsets $\Omega$ of a (potentially nonsmooth) space $X$ satisfying a curvature dimension condition $\CDD(K,N)$ with $K\in \mathbb R$ and $N>1$,
provided the topological boundary $\partial \Omega$ has a lower bound on its {\it inner mean curvature} in the sense of \cite{kettererHK}. The notion of inner mean curvature in \cite{kettererHK} is defined by means of the $1D$-localisation (needle decomposition) technique of Cavalletti and Mondino~\cite{cavmon} and coincides with the classical mean curvature of a hypersurface in the smooth context. We also assume that the boundary $\partial \Omega$ satisfies a measure theoretic regularity condition that is implied by an exterior ball condition. Hence, our result not only covers Kasue's theorem but also holds for a large class of domains in Alexandrov spaces or in Finsler manifolds. Kasue (and Li) were also able to prove a rigidity result analogous to Cheng's theorem~\cite{Cheng75} from the Bonnet--Myers context:
namely that, among smooth manifolds, their inscribed radius bound is obtained precisely
by the Euclidean unit ball. In the nonsmooth case there are also truncated cones that attain the maximal inradius; under an additional hypothesis known as~$\RCRD$,
 we prove that these are the only nonsmooth optimizers provided $\Omega$ is compact and its interior is connected.

To state our results first we recall the following definition.
For $\kappa\in \mathbb{R}$ we define $\cos_{\kappa}\colon \R\rightarrow \mathbb{R}$ as the solution of
\begin{gather}\label{trig ODE}
v''+\kappa v=0, \qquad \text{with} \quad v(0)=1 \qquad \text{and}\qquad v'(0)=0.
\end{gather}
The function $\sin_{\kappa}\colon \R \rightarrow \mathbb R$ is defined as solution of the same ordinary differential equation (ODE)
with initial values $v(0)=0$ and $v'(0)=1$. We define
\begin{gather}\label{pik}
\pi_{\kappa}:= \begin{cases} \dfrac{\pi}{\sqrt{\kappa}} & \mbox{if } \kappa>0, \\
\infty & \mbox{otherwise},
\end{cases}
\end{gather}
and $\tilde I_{\kappa}:=\overline{[0,\pi_\kappa)}$.
Let $K, H\in \mathbb{R}$ and $N> 1$. The Jacobian function is
\begin{gather}\label{Jacobian}
r\in \R\mapsto J_{K,H,N}(r):=\left(\cos_{K/{(N-1)}}(r) - \frac{H}{N-1}\sin_{K/(N-1)}(r)\right)_+^{N-1},
\end{gather}
where $(a)_+:=\max\{a, 0\}$ for $a\in \R$. Since $J_{K,H,N}(r)=J_{K,-H,N}(-r)$, its interval of positivity around the origin is given by $r \in (-r_{K,-H,N},r_{K,H,N})$, where
\begin{gather}\label{Jdomain}
r_{K,H,N} := \inf \{r\in (0,\infty)\colon J_{K,H,N}(r)=0\}.
\end{gather}
In \cite{Kasue83} and \cite{sakurai} the authors define
\begin{align}
\label{equ:ks}
s_{\kappa,\lambda}(r) = \cos_{\kappa}(r)- \lambda \sin_\kappa(r)
\end{align}
for $\kappa,\lambda\in \R$. They say the pair $(\kappa,\lambda)$ satisfies {\it the ball condition} if the equation $s_{\kappa, \lambda}(r)=0$ has a positive solution. The latter happens if and only if one of the following three cases holds: (1)~$\kappa>0$ and $\lambda\in \R$, (2)~$\kappa=0$ and $\lambda>0$ or (3)~$\kappa\leq 0$ and $\lambda>\sqrt{|\kappa|}$.
If $(\kappa,\lambda)=\big(\frac{K}{N-1},\frac{H}{N-1}\big)$, then~$r_{K,H,N}$ coincides with the smallest positive zero of $s_{\kappa,\lambda}$ if any exists;
moreover $s_{\kappa,\lambda}(r)<0$ for all $r>r_{K,H,N}$ if $\kappa \le 0$, while $s_{\kappa,\lambda}$ oscillates sinusoidally with mean zero and period
greater than $2 r_{K,H,N}$ if $\kappa>0$. In particular,
$r_{K,H,N}<\infty$ if and only if $\big(\frac{K}{N-1},\frac{H}{N-1}\big)$ satisfies the ball-condition.

For $\Omega \subset X$, letting $\Omega^c:= X \setminus \Omega$, our main theorem reads as follows:

\begin{Theorem}[inscribed radius bounds for metric measure spaces]\label{T:main}
Let $(X,d,\m)$ be an essentially nonbranching $\CDD(K',N)$ space with $K'\in \R$, $N\in (1,\infty)$ and $\supp \m=X$. Let $K, \kkapp\in \R$ such that $\big(\frac{K}{N-1},\kkapp\big)$ satisfies the ball condition. Let $\Omega\subset X$ be closed with {$\Omega\neq X$}, $\m(\Omega)>0$ and $\m(\partial \Omega)=0$ such that $\Omega$ satisfies the {\it restricted curvature-dimension condition} $\CDD_{r}(K,N)$ for $K\in \R$ $($Definition~{\rm \ref{def:cd})} and $\partial \Omega=S$ has finite inner curvature $($Definition~{\rm \ref{def:meancurvature})}. Assume the inner mean curvature $H_S^-$ satisfies $H^-_S\geq \kkapp (N-1)$ $\m_S$-a.e.\ where
$\m_S$ denotes the surface measure $($Definition~{\rm \ref{def:surfacemeasure})}.
Then
\begin{gather}\label{inradbound}
\inrad \Omega\leq r_{K,\kkapp(N-1),N},
\end{gather}
where $\inrad\Omega=\sup\limits_{x\in \Omega} d_{\Omega^c}(x)$ is the inscribed radius of $\Omega$.
\end{Theorem}
The theorem generalizes previous results for Riemannian manifolds \cite{Kasue83, martinli} and weighted Riemannian manifolds \cite{liwei_bakry_emery,liwei_rigidity, sakurai}. Moreover Theorem~\ref{T:main} also holds in the context of weighted Finsler manifolds and Alexandrov spaces and seems to be new in this context.

We also show:

\begin{Theorem}[stability]\label{T:main2}
Consider $(X,d,\m)$ and $\Omega\subset X$ as in the previous theorem. Then, for every $\epsilon>0$ there exists $\delta>0$ such that
\begin{gather*}
\inrad\Omega \leq r_{\bar K, \bar H, \bar N} +\epsilon
\end{gather*}
 provided $K\geq \bar K -\delta$, $H^-_S\geq \bar H -\delta$ $\m_S$-a.e.\ and $N\leq \bar N+\delta$ for $\bar K,\bar H\in \R $ and $\bar N\in (1,\infty)$.
 \end{Theorem}

\begin{Remark}[definitions and improvements] Let us comment on the definitions in Theorem~\ref{T:main} and generalizations.
\begin{enumerate}\itemsep=0pt
\item The curvature-dimension conditions $\CDD(K,N)$ and the restricted curvature-dimension condition $\CDD_{r}(K,N)$ for an essentially nonbranching metric measure space $(X,d,\m)$ are defined in Definition~\ref{def:cd}. If $(X,d,\m)$ satisfies the condition $\CDD(K,N)$ then $\Omega\neq \varnothing$ trivially satisfies $\CDD_r(K,N)$ for the same~$K$. For this we note that for essentially nonbranching $\CDD(K,N)$ spaces, $L^2$-Wasserstein geodesics between $\m$-absolutely continuous probability measures are unique \cite{Mon-Cav-17}.

\item Appendix \ref{S:MCP} extends the conclusions of Theorems~\ref{T:main} and~\ref{T:main2} to the case where the $\CDD(K,N)$ hypothesis is replaced by the measure contraction property $\MCP(K,N)$ proposed in \cite{ohtmea, stugeo2}, still under the essentially nonbranching hypothesis.

\item The {\em backward mean curvature bound} introduced in Appendix~\ref{S:backwards} also suffices for the conclusion of the above theorems, provided the finiteness assumed of the inner curvature of $\partial \Omega=S$ is replaced by the requirement that the surface measure $\m_{S_0}$ be Radon.
This alternate framework also suffices for the rigidity result of Theorem~\ref{T:main3} below.
It is related to but distinct from a notion presented in~\cite{cm_new}.

\item The property ``having finite inner curvature'' (Definition \ref{def:meancurvature}) rules out inward pointing cusps and cones, and is implied by an exterior ball condition for $\Omega$ (Lemma~\ref{lem:ballcondition}). The
surface measure $\m_S$ is defined in Definition~\ref{def:surfacemeasure}.

\item For $S$ with finite inner curvature, the definition of generalized inner mean curvature $H^-_S$ is given in Definition~\ref{def:meancurvature}.
Let us briefly sketch the idea. Using a needle decomposition associated to the signed distance function $d_S:=d_{\Omega}-d_{{\Omega}^c}$,
one can disintegrate the reference measure $\m$ into needles, meaning into conditional measures $\{\m_{\alpha}\}_{\alpha \in Q}$ (for a quotient space~$Q$) that are supported on curves $\gamma_{\alpha}$ of maximal slope of~$d_S$. For $\mathfrak q$-almost every curve $\gamma_{\alpha}$ with respect to the quotient measure $\mathfrak q$
of $\m$ on $Q$, there exists a {conditional} density $h_{\alpha}$ of $\m_{\alpha}$ with respect to the $1$-dimensional Hausdorff measure $\mathcal H^1$. Then the inner mean curvature for $\m_S$-a.e.\ $p=\gamma_{\alpha}(t_0)\in S$ is defined as $\frac{{\rm d}^-}{{\rm d}t} \log h_{\alpha}(t_0)=H_S^-(p)$. This left derivative quantifies the extent to which a given collection of needles are spreading (i.e., capturing more measure) as they exit $\Omega$. We postpone details to the Sections~\ref{subsec:1D} and~\ref{subsec:meancurvature}.
In the case $(X,d,\m)=(M,d_g,\vol_g)$ for a Riemannian manifold $(M,g)$ and~$\partial \Omega$ is a hypersurface the inner mean curvature coincides with the classical mean curvature.

\item Our assumptions cover the case of a Riemannian manifold with boundary: If $(X,d,\m)=(M,d_g,\vol_g)$ for a $n$-dimensional Riemannian manifold $(M,g)$ with boundary and Ricci lower bound $\ric_M\geq K$, then one can always construct a geodesically convex, $n$-dimensional Riemannian manifold $\tilde M$ with boundary such that $M$ isometrically embeds into $\tilde M$, and such that $\ric_{\tilde M}\geq K'$~\cite{wong}.
In particular, one can consider $M$ as a $\CDD_r(K,n)$ space that is a subset of the $\CDD(K',n)$ space $\big(\tilde M, d_{\tilde M}, \vol_{\tilde M}\big)$ (Remark~5.8 in~\cite{kettererHK}).
\end{enumerate}
\end{Remark}

\subsection{Cones and spherical suspension}
For smooth, $n$-dimensional Riemannian manifolds with non-negative Ricci curvature and boundary that has mean curvature bounded from below by $n-1$, equality in the inradius estimate is obtained precisely by the Euclidean unit ball (see \cite{Kasue83,liwei_rigidity,martinli}). In the nonsmooth case,
truncated cones also attain the maximal inradius.

Let $(X,d,\m)$ be a metric measure space.
 \begin{enumerate}\itemsep=0pt
\item The {\it Euclidean $N$-cone} over $(X,d,\m)$ is defined as the metric measure space
\begin{gather*}
\big( [0,\infty)\times X/\sim, d_{\rm Eucl}, \m^N_{\rm Eucl}\big) =: [0,\infty)\times^N_{\id} X,
\end{gather*}
where the equivalence relation $\sim$ is defined by $(s,x)\sim (t,y)$ if and only if either $s=t=0$ or $(s,x)=(t,y)$.
The tip of the cone is denoted by $o$. The distance $d_{\rm Eucl}$ is defined by
\begin{gather*}
d_{\rm Eucl}^2((t,x),(s,y)):=t^2+s^2 -2ts \cos [d(x,y)\wedge \pi ],
\end{gather*}
where $a \wedge b := \min\{a,b\}$, and the measure $\m^N_{\rm Eucl}$ is given by $r^N{\rm d}r \otimes \dm$.

If an Euclidean $N$-cone $[0,\infty)\times^N_{\id} X$ is a manifold then $X=\mathbb S^n$, $N=n\in \N$ and $[0,\infty)\times^n_{\id}\mathbb S^n$ is isometric to $\R^{n+1}$.
\item
The {\it hyperbolic $N$-cone} is defined similarly:
\begin{gather*}
\big( [0,\infty)\times X/\sim, d_{\rm Hyp}, \m^N_{\rm Hyp}\big) =: [0,\infty)\times_{\sinh}^N X.
\end{gather*}
The distance $d_{\rm Hyp}$ is defined by
\begin{gather*}
\cosh d_{\rm Hyp} ((t,x),(s,y)):=\cosh t \cosh s- \sinh t\sinh s \cos [d(x,y)\wedge \pi ],
\end{gather*}
and the measure $\m^N_{\rm Hyp}$ is given by $\sinh^Nr\,{\rm d}r \otimes \dm$.

If a hyperbolic $N$-cone $[0,\infty)\times_{\sinh}^N X$ is a manifold then $X=\mathbb S^n$, $N=n$ and $[0,\infty)\times_{\sinh}^n \mathbb S^n$ is isometric to the $n$-dimensional hyperbolic plane.

\item
Similar, the {\it spherical $N$-suspension} over $(X,d,\m)$ is defined as the metric space
\begin{gather*}
\big( [0,\pi]\times X/\sim, d_{\rm Susp}, \m^N_{\rm Susp}\big) =: [0,\pi]\times_{\sin}^N X,
\end{gather*}
where the equivalence relation $\sim$ is now defined by $(s,x)\sim (t,y)$ if and only if either $s=t \in \{0,\pi\}$ or $(s,x)=(t,y)$. The distance $d_{\rm Susp}$ is defined by
\begin{gather*}
\cos d_{\rm Susp}((t,x), (s,y)):= \cos t \cos s + \sin t \sin s\cos[d(x,y)\wedge \pi],
\end{gather*}
and the measure $\m_{\rm Susp}^N$ is given by $\sin^N t \,{\rm d}t \otimes \dm$.

If a sphercial $N$-cone $[0,\pi]\times_{\sin}^N X$ is a manifold then $X=\mathbb S^n$, $N=n\in \N$ and $[0,\pi]\times_{\sin}^n \mathbb S^n$ is isometric to the $(n+1)$-dimensional standard sphere $\mathbb S^{n+1}$.
\end{enumerate}

In each case $0$ may also be used to denote the equivalence class of the points $(0,x)$.

The next result shows that in an appropriate setting, cones and suspensions are the only maximizers of our inscribed radius bound.
 This requires the {\it Riemannian curvature-dimension condition} $\RCRD(K,N)$ (Definition~\ref{def:RCD}), a
 strengthening of the curvature-dimension condition that rules out Finsler manifolds and yields isometric rigidity theorems for metric measure spaces. This condition is crucial in the proof of the next theorem, since it permits us to exploit
 the volume cone rigidity theorem by DePhilippis and Gigli~\cite{DGi}.
 For $K>0$ one can view Theorem~\ref{T:main3} as a~version of the maximal diameter theorem~\cite{ketterer2}
 adapted to mean convex subsets of $\RCRD$ spaces.
 {For rigidity results pertaining to other inequalities in nonsmooth or even discrete settings, see recent work of Ketterer \cite{Ketterer15}, Nakajima, Shioya \cite{NakajimaShioya19} and Cushing et al.~\cite{CushingKamtueKoolenLiuMunchPeyerinhoff20}.}

\begin{Theorem}[rigidity]\label{T:main3}\sloppy
Let $(X,d,\m)$ be $\RCRD(K,N)$ for $K\in \R$ and $N\in (1,\infty)$ and let $\Omega\subset X$ be closed with {$\Omega\neq X$}, $\m(\Omega)>0$, connected and non-empty interior $\Omega^{\circ}$, and $\m(\partial \Omega)=0$. We assume that $K\in \{N-1, 0, -(N-1)\}$, $\partial \Omega =S$ has finite inner curvature, $S\neq \{pt\}$, and the inner mean curvature $\m_S$-a.e.\ satisfies $H_S^-\geq \kkapp (N-1) \in \R$.
 Then, there exists $x\in X$ such that
 \[ d_{\Omega^c}(x)=\inrad \Omega =r_{K,\kkapp(N-1),N}\] if and only if $r_{K,\chi(N-1), N}<\infty$ and
 there exists an $\RCRD(N-2,N-1)$ space $Y$ such that~$\Omega^\circ$ becomes isometric to the ball
 $B_{r_{K,\kkapp(N-1),N}}(0)$ of radius $r_{K,\kkapp(N-1),N}$ around the cone tip in \mbox{$\tilde I_{\frac{K}{N-1}} \times_{\sin_{K/(N-1)}}^{N-1} Y$},
 when each is equipped with the induced intrinsic distance which it inherits from its ambient space.
Here $\tilde I_{\frac{K}{N-1}}:=\overline{\big[0,\pi_\frac{K}{N-1}\big)}$.
\end{Theorem}
The theorem generalizes corresponding rigidity results for Riemannian manifolds \cite{Kasue83, martinli} and weighted Riemannian manifolds \cite{liwei_bakry_emery,liwei_rigidity, sakurai}.

\begin{Remark}[easy direction]
In the above rigidity theorem one direction is obvious. Let us explore this just for the case $K=0$ and $\kkapp=1$.

Let $Y$ be an $\RCRD(N-2,N-1)$ space. Both the Euclidean $N$-cone $X=[0,\infty) \times^N_{\id} Y$ and its truncation
$\Omega = [0,1]\times^N_{\id} Y$ are geodesically convex and satisfy $\RCRD(0,N)$~\cite{ketterer}.
The distance function $d_{\Omega^c}$ from $(X,d,\m)$ restricted to $\Omega$ is given by $d_{\Omega^c}(t,x)=1-t$.
In particular, $r_{0,N-1,N}=d_{\Omega^c}((0,x))=d_{\Omega^c}(o)=1$.

Moreover, $S$ has (inner) mean curvature equal to $N-1$ in the sense of Definition~\ref{def:meancurvature} in $X=[0,\infty)\times^{N-1}_{\id} Y$.
Indeed, we can see that points $(s,x)$ and $(t,y)$ in $\Omega$ lie on the same needle if and only if either $x=y$ or $st=0$. Hence, the needles in $\Omega$ for the corresponding $1D$-localization are $t\in (0,1)\mapsto\gamma(t)=(1-t,x)$, $x\in Y$. One can also easily check that $h_\gamma^{{1}/{N-1}}(t)=t$ for all needles $\gamma$ in the corresponding disintegration of $\m|_{\Omega}$. Hence $H_{\partial \Omega}^-\equiv N-1$.
\end{Remark}

\begin{Remark}
We may allow $S=\{pt\}$ in Theorem~\ref{T:main3} if we adopt the convention that $Y=\{pt\}$ satisfies $\RCRD(N-2,N-1)$ for all $N>1$.
This would correspond to case~(1) of De~Philippis and Gigli's classification, Theorem~\ref{thm:gp} below.
\end{Remark}

\section{Preliminaries}
\subsection{Curvature-dimension condition}

Let $(X,d)$ be a complete and separable metric space and let $\m$ be a locally finite Borel measure. We call $(X,d,\m)$ a metric measure space. We always assume $\supp \m=X$ and {$X\neq \{pt\}$}.

The length of a continuous curve $\gamma\colon [a,b]\rightarrow X$ is $L(\gamma)=\sup\big\{\sum_{i=1}^{k-1} d(\gamma(t_i),\gamma(t_{i+1}))\big\}\in [0,\infty]$ where the supremum is w.r.t.\ any subdivision of $[a,b]$ given by $a= t_1\leq t_2 \leq \dots \leq t_{k-1} \leq t_k=b$ and $k\in \mathbb{N}$.
Obviously $L(\gamma) \ge d(\gamma(a),\gamma(b))$;
a geodesic refers to any continuous curve $\gamma\colon [a,b]\rightarrow X$ saturating this bound.
We denote the set of constant speed geodesics $\gamma\colon [0,1]\rightarrow X$ with~$\mathcal G(X)$;
these are characterized by the identity
\[d(\gamma_s,\gamma_t) = (t-s) d(\gamma_0,\gamma_1)\]
for all $0\le s\le t \le 1$.
For $t\in [0,1]$ let $e_t\colon \gamma\in\mathcal G(X) \mapsto \gamma(t)$ be the evaluation map.
A subset of geodesics $F\subset \mathcal{G}(X)$ is said to be {\it nonbranching} if for any two geodesics $\gamma, \hat \gamma \in F$
such that there exists $\epsilon\in (0,1)$ with $\gamma|_{(0,\epsilon)}=\hat \gamma|_{(0,\epsilon)}$, it follows that $\gamma=\hat \gamma$.

\begin{Example}[Euclidean geodesics]When $X \subset \R^n$ is convex and $d(x,y)=|x-y|$ then $\mathcal G(X)$ consists of the affine maps $\gamma\colon [0,1]\to X$.
\end{Example}

The set of (Borel) probability measures on $(X,d,\m)$ is denoted with $\mathcal P(X)$, the subset of probability measures with finite second moment is~$\mathcal P^2(X)$, the set of probability measures in~$\mathcal P^2(X)$ that are $\m$-absolutely continuous is denoted with $\mathcal P^2(X,\m)$ and the subset of measures in $\mathcal P^2(X,\m)$ with bounded support is denoted with $\mathcal{P}_b^2(X,\m)$.

The space $\mathcal P^2(X)$ is equipped with the $L^2$-Wasserstein distance $W_2$, e.g.,~\cite{viltot}.
A dynamical optimal coupling is a~probability measure $\Pi\in \mathcal P(\mathcal G(X))$ such that $t\in [0,1]\mapsto (e_t)_{\#}\Pi$ is a~$W_2$-geodesic in $\mathcal P^2(X)$ where $(e_t)_{\#}\Pi$ denotes the push-forward under the map $\gamma\mapsto e_t(\gamma):=\gamma(t)$. The set of dynamical optimal couplings $\Pi\in \mathcal P(\mathcal G^{}(X))$ between $\mu_0,\mu_1\in \mathcal P^2(X)$ is denoted with $\OptGeo(\mu_0,\mu_1)$.

A metric measure space $(X,d,\m)$ is called \textit{essentially nonbranching} if for any pair $\mu_0,\mu_1\in \mathcal P^2(X,\m)$ any $\Pi\in \OptGeo(\mu_0,\mu_1)$ is concentrated on a set of nonbranching geodesics.

\begin{Definition}[distortion coefficients] For $K\in \mathbb{R}$, $N\in (0,\infty)$ and $\theta\geq 0$ we define the \textit{distortion coefficient} as
\begin{gather*}
t\in [0,1]\mapsto \sigma_{K,N}^{(t)}(\theta):=\begin{cases}
 \dfrac{\sin_{K/N}(t\theta)}{\sin_{K/N}(\theta)}\ &\mbox{if } \theta\in [0,\pi_{K/N}),\\
 \infty\ & \mbox{otherwise},
 \end{cases}
\end{gather*}
where $\pi_\kkapp:=
 \infty \ \textrm{if } \kkapp\le 0$ and $\pi_\kkapp
:=\frac{\pi}{\sqrt \kkapp}\ \textrm{if } \kkapp> 0$.
Here $\sin_{K/N}$ was defined after \eqref{trig ODE}, and $\sigma_{K,N}^{(t)}(0)=t$.
Moreover, for $K\in \mathbb{R}$, $N\in [1,\infty)$ and $\theta\geq 0$ the \textit{modified distortion coefficient} is defined as
\begin{gather*}
t\in [0,1]\mapsto \tau_{K,N}^{(t)}(\theta):=
 t^{\frac{1}{N}}\big[\sigma_{K,N-1}^{(t)}(\theta)\big]^{1-\frac{1}{N}},
 \end{gather*}
where our conventions are $0\cdot \infty :=0$ and $\infty^0:=1$.
\end{Definition}

\begin{Definition}[curvature-dimension conditions \cite{lottvillani, stugeo2}]\label{def:cd} An essentially nonbranching metric measure space $(X,d,\m)$ satisfies the \textit{curvature-dimension condition} $\CDD(K,N)$ for $K\in \mathbb{R}$ and $N\in [1,\infty)$ if for every $\mu_0,\mu_1\in \mathcal{P}_b^2(X,\m)$
there exists a dynamical optimal coupling $\Pi$ between~$\mu_0$ and~$\mu_1$ such that for all $t\in (0,1)$
\begin{gather}
\rho_t(\gamma_t)^{-\frac{1}{N}}\geq \tau_{K,N}^{(1-t)}(d(\gamma_0,\gamma_1))\rho_0(\gamma_0)^{-\frac{1}{N}}\nonumber\\
\hphantom{\rho_t(\gamma_t)^{-\frac{1}{N}}\geq}{} +\tau_{K,N}^{(t)}(d(\gamma_0,\gamma_1))\rho_1(\gamma_1)^{-\frac{1}{N}}
\qquad \mbox{for} \quad \Pi\mbox{-a.e.}\ \gamma\in \mathcal{G}(X),\label{cdcondition}
\end{gather}where $(e_t)_{\#}\Pi=\rho_t\m$.

Now instead suppose $(X,d,\m)$ satisfies $\CDD(K',N')$ for some $K' \in \R$ and $N' \ge 1$. We say a~subset
$\Omega\subset X$ with $\m(\Omega)>0$ satisfies the {\it restricted curvature-dimension condition} $\CDD_{r}(K,N)$ if for every dynamical optimal coupling $\Pi$ with $(e_t)_{\#}\Pi(\Omega)=1$ for all $t\in [0,1]$
and $\mu_0,\mu_1\in \mathcal{P}_b^2(X,\m)$, \eqref{cdcondition} holds for all $t\in [0,1]$.
\end{Definition}

\begin{Remark}[locally compact geodesic spaces]\label{R:Heine-Borel}
A $\CDD(K,N)$ space $(X,d,\m)$ for $N\in [1,\infty)$ is geodesic and locally compact. Hence, by the metric Hopf--Rinow theorem the space is proper (i.e., Heine--Borel) \cite[Theorem~2.5.28]{bbi}.
\end{Remark}

\subsection{Disintegration of measures} For further details about the content of this section we refer to \cite[Section~452]{fremlin}.

Let $(R,\mathcal R)$ be a measurable space, and let $\mathfrak Q\colon R\rightarrow Q$ be a map for a set $Q$. One can equip~$Q$ with the $\sigma$-algebra $\mathcal Q$ that is induced by $\mathfrak Q$ where $B\in \mathcal Q$ if $\mathfrak Q^{-1}(B)\in \mathcal{R}$. Given a measure $\m$ on $(R,\mathcal R)$, one can define its quotient measure $\mathfrak q$ on $Q$ via the push-forward $\mathfrak Q_{\#} \m=: \mathfrak q$.

\begin{Definition}[disintegration of measures]\label{D:disintegration}
A disintegration of a probability measure $\m$ that is consistent with $\mathfrak Q$ is a map $(B,\alpha)\in \mathcal R \times Q\mapsto \m_{\alpha}(B)\in [0,1]$ such that it follows
\begin{itemize}\itemsep=0pt
\item $\m_{\alpha}$ is a probability measure on $(R,\mathcal R)$ for every $\alpha\in Q$,
\item $\alpha\mapsto \m_{\alpha}(B)$ is $\mathfrak q$-measurable for every $B\in \mathcal R$,
\end{itemize}
and for all $B\in \mathcal R$ and $C\in \mathcal Q$ the {\it consistency condition}
\[
\m\big(B\cap \mathfrak Q^{-1}(C)\big)=\int_C \m_{\alpha}(B) \mathfrak q({\rm d}\alpha)
\]
holds. We use the notation $\{\m_{\alpha}\}_{\alpha\in Q}$ for such a~disintegration. We call the measures $\m_{\alpha}$ {\it conditional probability measures}
or {\it conditional measures}.
A~disintegration $\{\m_{\alpha}\}_{\alpha\in Q}$ consistent with $\mathfrak Q$
 is called strongly consistent if for $\mathfrak q$-a.e.\ $\alpha$ we have $\m_{\alpha}\big(\mathfrak Q^{-1}(\alpha)\big)=1$.
\end{Definition}

The following theorem is standard:

\begin{Theorem}[existence of unique disintegrations]
Assume that $(R,\mathcal R, \m)$ is a countably generated probability space and $R=\bigcup_{\alpha\in Q}R_{\alpha}$ is a partition of $R$. Let $\mathfrak Q\colon R\rightarrow Q$ be the quotient map associated to this partition, that is $\alpha =\mathfrak Q(x)$ if and only if $x\in R_{\alpha}$ and assume the corresponding quotient space $(Q,\mathcal Q)$ is a Polish space.

Then, there exists a strongly consistent disintegration $\{\m_{\alpha}\}_{\alpha\in Q}$ of $\m$ with respect to \mbox{$\mathfrak Q\colon R\!\rightarrow\! Q$} that is unique in the following sense: if $ \{\m'_{\alpha}\}_{\alpha\in Q}$ is another consistent disintegration of $\m$ {with respect to} $\mathfrak Q$ then $\m_{\alpha}=\m'_{\alpha}$ for $\mathfrak q$-a.e.\ $\alpha \in Q$.
\end{Theorem}

\subsection[1D-localization]{$\boldsymbol{1D}$-localization}\label{subsec:1D}

In this section we will recall the basics of the localization technique introduced by Cavalletti and Mondino for 1-Lipschitz
functions as a nonsmooth analog of Klartag's needle decomposition: needle refers to any geodesic along
which the Lipschitz function attains its maximum slope, also called transport rays here and by Klartag and others \cite{EvansGangbo99,FeldmanMcCann02,Klartag17}.
 The presentation follows Sections~3 and~4 in \cite{cavmon}. We assume familiarity with basic concepts in optimal transport (for instance~\cite{viltot}).

Let $(X,d,\m)$ be a proper metric measure space (with $\supp\m =X$ as we always assume).

Let $u\colon X\rightarrow \mathbb{R}$ be a $1$-Lipschitz function. Then the {\it transport ordering}
\begin{gather*}
\Gamma_u:=\{(x,y)\in X\times X\colon u(y)-u(x)=d(x,y)\}
\end{gather*}
is a $d$-cyclically monotone set, and one defines
$\Gamma_u^{-1}=\{(x,y)\in X\times X\colon (y,x)\in \Gamma_u\}$.

Note that we switch orientation in comparison to \cite{cavmon} where Cavalletti and Mondino define~$\Gamma_u$ as~$\Gamma_u^{-1}$.

The union $\Gamma_u \cup \Gamma_u^{-1}$ defines a relation $R_u$ on $X\times X$, and $R_u$ induces the {\it transport set with endpoint and branching points}
\[ \mathcal T_{u,e}:= P_1(R_u\backslash \{(x,y)\colon x=y\in X\})\subset X,\]
where $P_1(x,y)=x$. For $x\in \T_{u,e}$ one defines $\Gamma_u(x):=\{y\in X\colon (x,y)\in \Gamma_u\}$,
and similarly~$\Gamma_u^{-1}(x)$ and~$R_u(x)$. Since $u$ is $1$-Lipschitz,
 $\Gamma_u$, $\Gamma_u^{-1}$ and $R_u$ are closed, as are $\Gamma_u(x)$, $\Gamma_u^{-1}(x)$ and~$R_u(x)$.

The {\it forward} and {\it backward branching points} are defined respectively as
\begin{gather*}
A_{+} := \{x\in \mathcal T_{u,e}\colon \exists\, z,w\in \Gamma_u(x) \mbox{ \& } (z,w)\notin R_u\}, \\
 A_{-} := \{x\in \mathcal T_{u,e} \colon \exists\, z,w\in \Gamma_u^{-1}(x) \mbox{ \& } (z,w)\notin R_u\}.
\end{gather*}
Then one considers the {\it $($nonbranched$)$ transport set} as $\mathcal T_u:=\mathcal T_{u,e}\backslash (A_+ \cup A_-)$ and the {\it $($non\-bran\-ched$)$ transport relation} as the restriction of $R_u$ to $\mathcal T_u\times \mathcal T_u$.

The sets $\T_{u,e}$, $A_{+}$ and $A_-$ are $\sigma$-compact (\cite[Remark~3.3]{cavmon} and \cite[Lemma~4.3]{cavom} respectively), and $\T_u$ is a Borel set. In \cite[Theorem~4.6]{cavom} Cavalletti shows that the restriction of $R_u$ to $\mathcal T_u\times \mathcal T_u$ is an equivalence relation.
Hence, from $R_u$ one obtains a partition of $\mathcal T_u$ into a disjoint family of equivalence classes $\{X_{\alpha}\}_{\alpha\in Q}$. There exists a measurable section $s\colon \T_u\rightarrow \T_u$ \cite[Proposition~5.2]{cavom}, such that if $(x,s(x)) \in R_u$ and $(y,x)\in R_u$ then $s(x)=s(y)$, and $Q$ can be identified with the image of $\T_u$ under~$s$. Every $X_{\alpha}$ is isometric to an interval $I_\alpha\subset\mathbb{R}$ (cf.\ \cite[Lemma~3.1]{cavmon} and the comment after Proposition~3.7 in~\cite{cavmon}) via a distance preserving map $\gamma_{\alpha}\colon I_\alpha \rightarrow X_{\alpha}$ where $\gamma_\alpha$ is parametrized such that $d(\gamma_\alpha(t), s(\gamma_\alpha(t)))=\sgn(\gamma_\alpha(t))t$, $t\in I_\alpha$, and where $\sgn x$ is the sign of $u(x)-u(s(x))$. The map $\gamma_{\alpha}\colon I_\alpha\rightarrow X$ extends to a geodesic also denoted $\gamma_{\alpha}$ and defined on the closure $\overline{I}_{\alpha}$ of $I_{\alpha}$. We set $\overline{I}_{\alpha}=[a(X_{\alpha}),b(X_{\alpha})]$.

Then, the quotient map $\mathfrak Q\colon \T_u\rightarrow Q$ given by the section $s$ above is measurable, and we set $\mathfrak q:= \mathfrak Q_{\#}\m|_{\T_u}$. {Hence, we can and will consider $Q$ as a subset of $X$, namely the image of $s$, equipped with the induced measurable structure, and $\mathfrak q$ as a Borel measure on $X$. By inner regularity we replace $Q$ with a Borel set $Q'$ such that $\mathfrak q(Q\backslash Q')=0$ and in the following we denote $Q'$ by $Q$} (compare with \cite[Proposition~3.5]{cavmon} and the following remarks).

In \cite[Theorem 3.3]{cav-mon-lapl-18}, Cavalletti and Mondino extend Definition~\ref{D:disintegration} to disintegrate measures~$\m$ which are merely $\sigma$-finite
by using a positive function $f$ on $X$ to relate $\m$ to a probability measure~$f(x)\dm(x)$. Using the framework of this extension, which we also adopt, they prove:

\begin{Theorem}[disintegration into needles/transport rays]\label{T:CM disintegration}
Let $(X,d,\m)$ be a geodesic metric measure space with $\supp\m =X$ and $\sigma$-finite $\m$. Let $u\colon X\rightarrow \mathbb{R}$ be a $1$-Lipschitz function, let~$\{X_{\alpha}\}_{\alpha\in Q}$ be the induced partition of $\mathcal T_u$ via $R_u$, and let $\mathfrak Q\colon \T_u\rightarrow Q$ be the induced quotient map as above.
Then, there exists a unique {strongly consistent disintegration} $\{\m_{\alpha}\}_{\alpha\in Q}$ of $\m|_{\T_u}$ {with respect to} $\mathfrak Q$.
\end{Theorem}
Now, we assume that $(X,d,\m)$ is an essentially nonbranching $\CDD(K,N)$ space for $K\in \R$ and $N> 1$. Recall the Bishop--Gromov inequality holds and $\m$ is therefore $\sigma$-finite.
The following is \cite[Lemma~3.4]{cav-mon-lapl-18}.
\begin{Lemma}[negligibility of branching points]
\label{somelemma}
Let $(X,d,\m)$ be an essentially nonbranching $\CDD(K,N)$ space for $K\in \R$ and $N\in (1,\infty)$ with $\supp \m=X$.
Then, for any $1$-Lipschitz function $u\colon X\rightarrow \R$, it follows $\m(\T_{u,e}\backslash \T_u)=0$.
\end{Lemma}
The initial and final points are defined by
\begin{gather*}
\mathfrak a := \big\{x\in \T_{u,e}\colon \Gamma^{-1}_u(x)=\{x\}\big\}, \qquad
\mathfrak b := \{x\in \T_{u,e}\colon \Gamma_u(x)=\{x\} \}.
\end{gather*}
In \cite[Theorem~7.10]{cavmil} it was proved that under the assumption of the previous lemma there exists $\hat Q\subset Q$ with $\mathfrak q(Q\backslash \hat Q)=0$ such that for $\alpha\in \hat Q$ one has $\overline{X_\alpha}\backslash \T_u\subset \mathfrak a\cup \mathfrak b$. In particular, for $\alpha\in \hat Q$ we have
\begin{gather}\label{somehow}
R_u(x)=\overline{X_\alpha}\supset X_\alpha \supset (R_u(x))^{\circ} \qquad \forall\, x\in \mathfrak Q^{-1}(\alpha)\subset \T_u,
\end{gather}
where $(R_u(x))^\circ$ denotes the relative interior of the closed set $R_u(x)$.

The following is \cite[Theorem 3.5]{cav-mon-lapl-18}.
\begin{Theorem}[conditional measures inherit curvature-dimension bounds]\label{th:1dlocalisation}
Let $(X,d,\m)$ be an essentially nonbranching $\CDD(K,N)$ space with $\supp\m=X$, $K\in \R$ and $N\in (1,\infty)$.
For any $1$-Lipschitz function $u\colon X\rightarrow \R$, let $\{\m_{\alpha}\}_{\alpha\in Q}$ denoted the disintegration of $\m|_{\T_u}$
from Theo\-rem~{\rm \ref{T:CM disintegration}} which is strongly consistent with the quotient map $\mathfrak Q\colon \T_u \rightarrow Q$.
Then there exists $\tilde Q$ such that $\mathfrak q(Q\backslash \tilde Q)=0$ and $\forall\, \alpha\in \tilde Q$, $\m_{\alpha}$ is a Radon measure with $\dm_{\alpha}=h_{\alpha}\,{\rm d}\mathcal{H}^1|_{X_{\alpha}}$ and $(X_{\alpha}, d, \m_{\alpha})$ verifies the condition $\CDD(K,N)$.
More precisely, for all $\alpha\in \tilde Q$ it follows that
\begin{gather}\label{kuconcave}
h_{\alpha}(\gamma_t)^{\frac{1}{N-1}}\geq \sigma_{K/N-1}^{(1-t)}(|\dot\gamma|)h_{\alpha}(\gamma_0)^{\frac{1}{N-1}}+\sigma_{K/N-1}^{(t)}(|\dot\gamma|)h_{\alpha}(\gamma_1)^{\frac{1}{N-1}}
\end{gather}
for every affine map $\gamma\colon [0,1]\rightarrow (a(X_\alpha),b(X_\alpha))$.
\end{Theorem}

\begin{Remark}[semiconcave densities on needles] \label{rem:kuconcave} The property
\eqref{kuconcave} yields that $h_{\alpha}$ is locally Lipschitz continuous on $(a(X_\alpha),b(X_\alpha))$ \cite[Section 4]{cavmon},
and that $h_{\alpha}\colon (a(X_\alpha), b(X_\alpha)) \rightarrow (0,\infty)$ satisfies
\begin{gather}\label{vvv}
\frac{{\rm d}^2}{{\rm d}r^2}h_{\alpha}^{\frac{1}{N-1}}+ \frac{K}{N-1}h_{\alpha}^{\frac{1}{N-1}}\leq 0\qquad \mbox{on $(a(X_{\alpha}),b(X_{\alpha}))$} \mbox{ in distributional sense;}
\end{gather}
in particular, $h_{\alpha}^{\frac{1}{N-1}}$ is semiconcave on $(a(X_{\alpha}),b(X_{\alpha}))$, hence admits left and right derivatives at each point.

Conversely, it is well-known that any function $h_\alpha \ge 0$ satisfying \eqref{vvv} can be chosen to be continuous up to its endpoints and that this extension then satisfies~\eqref{kuconcave}; see also Remark~\ref{rem:extden}.
\end{Remark}
\begin{Remark}\label{rem_mer}
We observe the following from the proof of \cite[Theorem 4.2]{cavmon}: when $\Omega\subset X$ satisfies the restricted condition $\CDD_{r}(K',N)$ and the $1$-Lipschitz function $u=d_{\Omega^c}$ is chosen to be the distance function to $\Omega^c$, then $h_\gamma$ satisfies~\eqref{kuconcave} with $K'$ replacing~$K$.

Let us be a little bit more precise here. For the proof of Theorem~4.2 in~\cite{cavmon} the authors construct $L^2$-Wasserstein geodesics between $\m$-absolutely continuous probability measures such that the corresponding optimal dynamical plans are supported on transport geodesics of the $1$-Lipschitz function $\phi$ that appears in the statement of \cite[Theorem~4.2]{cavmon}.

In our situation, when $\phi$ is actually $d_{\Omega^c}$, all transport geodesics of positive length are inside of~$\Omega$. Hence, the $L^2$-Wasserstein geodesics constructed by Cavalletti and Mondino are concentrated in $\Omega$ and the restricted condition~$\CDD_r(K', N)$ applies. Then we can follow verbatim the proof of Theorem~4.2 in~\cite{cavmon}.
\end{Remark}

\begin{Remark}[extended densities]\label{rem:extden}
The Bishop--Gromov volume monotonicity implies that $h_\alpha$ can always be extended to continuous function on $[a(X_\alpha),b(X_\alpha)]$ \cite[Remark~2.14]{cav-mon-lapl-18}. Then~\eqref{kuconcave} holds for every affine map $\gamma\colon [0,1]\rightarrow [a(X_\alpha),b(X_\alpha)]$.
We set $\left(h_{\alpha}\circ \gamma_{\alpha}(r)\right)\cdot 1_{[a(X_{\alpha}),b(X_{\alpha})]}=h_{\alpha}(r)$
and consider $h_{\alpha}$ as function that is defined everywhere on $\R$.
We also consider $\frac{{\rm d}}{{\rm d}r}h_{\alpha}\colon X_{\alpha}\rightarrow \R$ defined a.e.\ via
$\frac{{\rm d}}{{\rm d}r}(h_{\alpha}\circ\gamma_{\alpha})(r)=:\frac{{\rm d}}{{\rm d}r}h_{\alpha}(r)$.

It is standard knowledge that the derivatives from the right and from the left
\begin{gather*}
\frac{{\rm d}^{+}}{{\rm d}r} h_\alpha(r)= \lim_{t\downarrow 0} \frac{h_{\alpha}(r+t)-h_\alpha(r)}{t},\qquad
\frac{{\rm d}^-}{{\rm d}r}h_\alpha(r)= \lim_{t\uparrow 0}\frac{h_{\alpha}(r+t)-h_\alpha(r)}{t}
 \end{gather*}
exist for $r\in [a(X_\alpha), b(X_\alpha))$ and $r\in (a(X_\alpha), b(X_\alpha)]$ respectively. Moreover, we set $\frac{{\rm d}^{+}}{{\rm d}r}h_\alpha= - \infty$ in $b(X_\alpha)$ and $\frac{{\rm d}^-}{{\rm d}r}h_\alpha=\infty$ in $a(X_\alpha)$.
\end{Remark}
\begin{Remark}[generic geodesics]\label{def:dagger}\label{dagger}
In the following we set $Q^\dagger:= \tilde Q \cap \hat Q$, where $ \tilde Q$ and $\hat Q$ index the transport rays identified between Lemma~\ref{somelemma}
and Theorem~\ref{th:1dlocalisation}.
 Then, $\mathfrak q\big(Q\backslash Q^\dagger\big)=0$ and for every $\alpha \in Q^\dagger$ the inequality \eqref{kuconcave} and \eqref{somehow} hold. We also set $\mathfrak Q^{-1}(Q^{\dagger})=:\T_u^{\dagger}\subset \T_u$ and $\bigcup_{x\in \T_u^{\dagger}} R_u(x)=:\T_{u,e}^{\dagger}\subset \T_{u,e}$.
\end{Remark}

\subsection{Generalized mean curvature}\label{subsec:meancurvature}

Let $(X,d,\m)$ be a metric measure space as in Theorem~\ref{th:1dlocalisation}.
Let $\Omega\subset X$ be a closed subset, and let $S=\partial \Omega$ such that $\m(S)=0$. The function $d_\Omega\colon X\rightarrow \mathbb{R}$ is given by
\begin{gather*}
d_\Omega(x) := \inf_{y\in {\Omega}} d(x,y).
\end{gather*}
The signed distance function $d_S$ for $S$ is given by
\begin{gather*}
d_S:=d_{{\Omega}}- d_{{\Omega^{c}}}\colon \ X\rightarrow \mathbb{R}.
\end{gather*}
It follows that $d_S(x)=0$ if and only if $x\in S$, and $d_S\leq 0$ if $x\in \Omega$ and $d_S\geq 0$ if $x\in \Omega^c$.
It is clear that $d_S|_{\Omega}= -d_{\Omega^{c}}$ and $d_S|_{\Omega^c}=d_\Omega$.
Setting $v=d_S$ we can also write
\[
d_S(x)= \sign(v(x))d(\{v=0\},x), \qquad \forall\, x\in X.
\]
Since $X$ is proper, $d_S$ is $1$-Lipschitz \cite[Remarks~8.4 and~8.5]{cav-mon-lapl-18}. Let $\Omega^\circ$ denote the topological interior of $\Omega$.

Let $\mathcal{T}_{d_S,e}$ be the transport set of $d_S$ with end- and branching points. We have $\mathcal{T}_{d_S, e}\supset X\backslash S$. In particular, we have $\m(X\backslash \T_{d_S})=0$ by Lemma \ref{somelemma} and $\m(S)=0$.
Therefore, by Theorem \ref{th:1dlocalisation} the $1$-Lipschitz function $d_S$ induces a partition $\left\{X_{\alpha}\right\}_{\alpha\in Q}$ of $(X,d,\m)$ up to a set of measure zero for a measurable quotient space $Q$, and a disintegration $\{m_{\alpha}\}_{\alpha\in Q}$ that is strongly consistent with the partition. The subset $X_{\alpha}$, $\alpha \in Q$, is the image of a distance preserving map $\gamma_{\alpha}\colon I_\alpha\rightarrow X$ for an interval $I_\alpha\subset \mathbb R$ with $\overline{I}_\alpha= [a(X_\alpha), b(X_\alpha)]\ni 0$.

We consider $Q^\dagger\subset Q$ as in Remark~\ref{dagger}. One has the representation
\begin{gather*}
\m(B)=\int_{Q} \m_{\alpha}(B) \,{\rm d}\mathfrak q(\alpha)=\int_{Q^{\dagger}} \int_{\gamma_{\alpha}^{-1}(B)} h_{\alpha}(r) \,{\rm d}r \, {\rm d}\mathfrak q(\alpha)
\end{gather*}
for all Borel $B \subset X$.

For any transport ray $X_{\alpha}$ with $\alpha\in Q^{\dagger}$, it follows that $d_S(\gamma_\alpha(b(X_{\alpha})))\geq 0$ and
$d_S(\gamma_\alpha(a(X_{\alpha})))\allowbreak \leq 0$ (for instance compare with \cite[Remark~4.12]{cav-mon-lapl-18}).

\begin{Remark}[measurability and zero-level selection]\label{R:zero-level selection}
It is easy to see that $A:=\mathfrak Q^{-1}(\mathfrak Q(S\cap \T_{d_S}))\subset \T_{d_S}$ is a measurable subset.
The {\it reach}
$A\subset \T_{d_S}$ is defined such that $\forall\, \alpha \in \mathfrak Q(A)$ we have $X_\alpha\cap S =\{\gamma(t_\alpha)\}\neq \varnothing$ for a unique $t_\alpha\in I_{\alpha}$. Then, the map ${\hat s}\colon \gamma(t)\in A \mapsto \gamma(t_\alpha)\in S\cap \T_{d_S}$ is a measurable section (i.e., selection) on $A\subset \T_{d_S}$, and one can identify the measurable set $\mathfrak Q(A)\subset Q$ with $A\cap S$ and can parameterize $\gamma_\alpha$ such that $t_\alpha=0$.

This measurable section $\hat s$ on $A$ is fixed for the rest of the paper. The reach $A$ is the union of all disjoint needles that intersect with $\partial \Omega$ -- possibly in $a(X_\alpha)$ (or in $b(X_\alpha)$) provided $a(X_\alpha)$ (respectively $b(X_\alpha)$) belongs to $I_\alpha$. We shall also define the {\it inner reach} $B_{\rm in}$ as the union of all needles disjoint from
$\overline{\Omega^c}$ and the {\it outer reach} $B_{\rm out}$ as the union of all needles disjoint from $\Omega$. The superscript $\dagger$ will be used indicate intersection with $\mathcal T^{\dagger}_{d_S}$. Thus
\begin{gather*}
{A}\cap \mathfrak \T_{d_S}^{\dagger}=: A^{\dagger} \qquad \mbox{and}\qquad \bigcup_{x\in A^{\dagger}}R_{d_S}(x)=:A_e^{\dagger}.
\end{gather*}
The sets $A^\dagger$ and $A_e^\dagger$ are measurable, and also
\begin{gather}\label{Binout}
B_{\rm in}^\dagger:=\Omega^\circ \cap \T^\dagger_{d_S} \backslash A^\dagger\subset \T^\dagger_{d_S} \qquad \text{and} \qquad B^\dagger_{\rm out}:=\Omega^c \cap \T^\dagger_{d_S} \backslash A^\dagger\subset \T_{d_S}
\end{gather}
as well as $\bigcup_{x\in B^\dagger_{\rm out}} R_{d_S}(x)=:B^\dagger_{{\rm out}, e}$ and $\bigcup_{x\in B^\dagger_{\rm in}}R_{d_S}(x)=:B^\dagger_{{\rm in},e}$
are measurable.

The map $\alpha \in \mathfrak Q(A^\dagger)\mapsto h_\alpha(0)\in \R$ is measurable (see \cite[Proposition~10.4]{cavmil}).
\end{Remark}

\begin{Definition}[surface measures]\label{def:surfacemeasure}
Taking $S=\partial \Omega$ as above, we use the disintegration of Remark~\ref{R:zero-level selection} to define the {\it surface measure} $\m_S$ via
\[
\int\phi(x)\dm_S(x):=\int_{\mathfrak Q(A^{\dagger})} \phi(\gamma_\alpha(0)) h_{\alpha}(0) \,{\rm d}\mathfrak q(\alpha)
\]
for any bounded and continuous function $\phi\colon X\rightarrow \R$.
That is, $\m_S$ is the push-forward of the measure {$h_\alpha(0) \,{\rm d}\mathfrak q(\alpha)|_{\mathfrak Q(A^{\dagger})}$} under the map $\hat s\colon \gamma\in \mathfrak Q(A^{\dagger})\mapsto \gamma(0) \in S$.
\end{Definition}

 \begin{Remark}[surface measure via ray maps] \label{rem:surmea}Let us briefly explain the previous definition from the viewpoint of the ray map \cite[Definition~3.6]{cavmon}
or its precursor from the smooth setting~\cite{FeldmanMcCann02}.
For the definition we fix a measurable extension $ s_0\colon \mathcal T_{d_S}\rightarrow \mathcal T_{d_S}$ such that $ s_0|_{A^{\dagger}}=\hat s$ as in Remark~\ref{R:zero-level selection}. As was explained in Section~\ref{subsec:1D} such a {section} allows us to identify the quotient space~$Q$ with a Borel subset in~$X$ up to a set of $\mathfrak q$-measure~$0$.
 Following \cite[Definition~3.6]{cavmon} we define the ray map
 \begin{gather*}
 g\colon \ \mathcal V\subset \mathfrak Q(A\cup B_{\rm in})\times (-\infty,0] \rightarrow X
 \end{gather*}
into $\Omega$ and its domain $\mathcal V$ via its graph
 \begin{gather*}
 \mbox{graph}(g) =\{(\alpha,t,x)\in \mathfrak Q(A)\times\R\times \Omega\colon x\in X_\alpha, -d(x,\alpha)=t\}\\
\hphantom{\mbox{graph}(g) =}{} \cup \{(\alpha,t,x)\in \mathfrak Q(B_{\rm in}) \times \R\times \Omega\colon x\in X_\alpha, -d(x, \gamma_\alpha(b(X_\alpha)))=t\}.
 \end{gather*}
 This is exactly the ray map as in \cite{cavmon} up to a reparametrisation for $\alpha\in \mathfrak Q(B_{\rm in})$.
Note that $g(\alpha,0)=\gamma_\alpha(0)=\alpha$ and $g(\alpha,t)=\gamma_\alpha(t)$ if $\alpha\in \mathfrak Q(A)$ but $\gamma_\alpha(t+d(b(X_\alpha),\alpha))=g(\alpha,t)$ for $\alpha\in \mathfrak Q(B_{\rm in})$.
 Then the disintegration for a non-negative $\phi\in C_b(\Omega)$ takes the form
 \begin{gather*}
 \int_{\Omega} \phi\thinspace \dm = \int_Q \int_{\mathcal V_\alpha} \phi\circ g(\alpha,t) h_\alpha\circ g(\alpha,t) \,{\rm d} \mathcal L^1{ (t)} \,{\rm d}\mathfrak q(\alpha),
 \end{gather*}
where $\mathcal V_\alpha = P_2(\mathcal V \cap \{\alpha\}\times \R)\subset \R$ and $P_2(\alpha,t)=t$. With Fubini's theorem the right hand side is
 \begin{gather*}
 \int_{\mathcal V} \phi \circ g(\alpha,t) h_\alpha \circ g(\alpha,t) \,{\rm d}\big(\mathfrak q\otimes \mathcal L^1\big)(\alpha,t)= \iint_{\mathcal V_t} \phi \circ g(\alpha, t) h_\alpha\circ g(\alpha, t) \,{\rm d}\mathfrak q(\alpha) \,{\rm d}t,
 \end{gather*}
 where $\mathcal V_t= P_1(\mathcal V \cap Q\times \{t\})\subset Q$ and $P_1(\alpha,t)=\alpha$. In particular, for $\mathcal L^1$-a.e.\ $t\in \R$ the map $\alpha \mapsto h_{\alpha}\circ g(\alpha,t)$ is measurable.
 Hence, for $\mathcal L^1$-a.e.\ $t\in \R$ we define ${\rm d}\mathfrak p_t (\alpha)= h_\alpha \circ g(\alpha,t) \,{\rm d} \mathfrak q|_{\mathcal V_t}(\alpha)$ on $Q$. Then disintegration takes the form
 \begin{align*}
 \m|_{\Omega}=\m|_{\Omega\cap \mathcal T_{d_S}}= \int (g(\cdot,t)_{\#} \mathfrak p_t) \,{\rm d}t.
 \end{align*}
Now, we can consider the push-forward $\m_{S_0}= g(\cdot,0)_{\#} \mathfrak p_0$. When $B_{\rm in}\neq \varnothing$ then $\m_{S_0}$ may be concentrated on a~larger set than $\m_{S}$ but by construction one recognizes that $\m_{S}=\m_{S_0}|_{A^{\dagger}}$.
\end{Remark}

\begin{Definition}[inner mean curvature]\label{def:meancurvature}
Set $S=\partial \Omega$ and let $\{X_{\alpha}\}_{\alpha\in Q}$ be the disintegration induced by $u:=d_S$.
Recalling \eqref{Binout}, we say that $S$ has {\em finite inner} (respectively {\em outer}) {\em curvature} if $\m\big(B^\dagger_{\rm in}\big)=0$ (respectively $\m\big(B^\dagger_{\rm out}\big)=0$),
and $S$ has {\em finite curvature} if $\m\big(B^\dagger_{\rm out}\cup B^\dagger_{\rm in}\big)=0$.
If~$S$ has finite inner curvature we define the {inner mean curvature} of $S$ {$\m_S$-almost everywhere} as
\begin{gather*}
p\in S\mapsto
H_S^-(p):= \begin{cases}
\dfrac{{\rm d}^-}{{\rm d}r}|_{r=0}\log h_\alpha\circ \gamma_\alpha(r)
& \mbox{if }
\ p=\gamma_\alpha(0)\in S\cap A^{\dagger},\\
\infty& \mbox{if } \ p\in B^\dagger_{{\rm out}, e}\cap S,
\end{cases}
\end{gather*}
where we set
$\frac{{\rm d}^-}{{\rm d}r} \log h_{\alpha}(\gamma_\alpha(0))= -\infty$ if $h_\alpha(\gamma_\alpha(0))=0$.
\end{Definition}
\begin{Remark}[(sign) conventions]
We point out two differences in comparison to~\cite{kettererHK}:
For the definition of $A^{\dagger}$ we do not remove points that lie in $\mathfrak a$ and $\mathfrak b$, and we switched signs in the definition of inner mean curvature. The latter allows us to work with mean curvature bounded below instead of bounded above.
\end{Remark}

\begin{Remark}[smooth case]
Let us briefly address the case of a Riemannian manifold $(M,g)$ equipped with a measure of the form $\dm=\Psi{\rm d}\!\vol_g$ for $0<\Psi\in C^{\infty}(M)$ and $\Omega$ with a boundary~$S$ which is a smooth compact submanifold.
For every $x\in S$ there exist $a_x<0$ and $b_x>0$ such that $\gamma_x(r)=\exp_x(r\nabla d_S(x))$ is a minimal geodesic on $(a_x,b_x)\subset \R$, and we define
\begin{gather*}
\mathcal U= \{(x,r)\in S\times \R\colon r\in (a_x,b_x) \}\subset S\times \R
\end{gather*}
and the map $T\colon \mathcal U \rightarrow M$ via $T(x,r)=\gamma_x(r)$. The map $T$ is a diffeomorphism on $\mathcal U$, with $\vol_g(M\backslash T(\mathcal U))=0$ and the integral
of $\phi \in C(M)$ can be computed effectively by the following formula:
\begin{gather*}
\int \phi \dm
 = \int_S \int_{a_x}^{b_x} \phi \circ T(x,r) \det DT_{(x,r)}|_{T_xS} \Psi\circ T(x,r) \,{\rm d}r\, {\rm d}\!\vol_S(x),
\end{gather*}
where $\vol_S$ is the induced Riemannian surface measure on~$S$.
By comparison with the needle technique disintegration it is not difficult to see that $\dm_S=\Psi \,{\rm d}\!\vol_S$. Moreover, the open needles for $d_S$ are the geodesics $\gamma_x\colon (a_x,b_x)\rightarrow M$ and the densities $h_x(r)$ are given by $c(x)\det DT_{(x,r)} \Psi\circ T(x,r)$ for some normalization constant $c(x)$, $x\in S$.

A direct computation then yields
\begin{gather*}
\frac{{\rm d}}{{\rm d}r} \log h_x(0) = H_S(x) + \langle \nabla d_S(x),\nabla \log \Psi \rangle(x), \qquad \forall\, x\in S,
\end{gather*}
where $H_S$ is the standard mean curvature, i.e., the trace of the second fundamental form of~$S$.
\end{Remark}
\begin{Definition}[exterior ball condition]
Let $\Omega\subset X$ and $\partial \Omega =S$. Then $S$ satisfies the exterior ball condition if for all $x\in S$ there exists $r_{x}>0$ and $p_x\in \Omega^c$ such that $d(x,p_x)=r_{x}$ and $B_{r_{x}}(p_x)\subset \Omega^c$.
We say $S$ satisfies a uniform exterior ball condition if there exists $\delta>0$ such that $r_x\geq \delta$ for all $x\in S$.
\end{Definition}

\begin{Lemma}[exterior ball criterion for finite inner curvature] \label{lem:ballcondition}
Let $\Omega \subset X$. If $S=\partial \Omega$ satisfies the exterior ball condition, then $S$ has finite inner curvature.
\end{Lemma}

\begin{proof}Let $S$ satisfy the exterior ball condition. Then for every $x\in S$ there exists a point $p_x\in \Omega^c$ and a geodesic $\gamma_x\colon [0,r_x]\rightarrow \Omega^c$ from $x$ to $p_x$ such that $L(\gamma_x)=d(x,p_x)=r_x$ and $d(p_x, y)>r_x$ for any $y\in S\backslash \{x\}$. Hence, $d_S(p_x)=r_x$ and the image of $\gamma_x$ is contained in $R_{d_S}(x)$.

Recall the definition of $Q^\dagger\subset Q$ (Remark~\ref{def:dagger}). Since $Q^\dagger$ has full $\mathfrak q$-measure, it is enough to show that for all $\alpha\in Q^{\dagger}$ the endpoint $b(X_\alpha)>0$. Then also $B_{\rm in}^{\dagger}=\varnothing$. Assume the contrary.
Let $\alpha'\in Q^\dagger$ and let $\gamma':=\gamma_{\alpha'}$ be the corresponding geodesic such that $b(X_{\alpha'})=0$, that is $\mbox{Im}(\gamma'|_{(a(X_{\alpha'}),0)})\subset \Omega$. The concatenation $\gamma''\colon (a(X_{\alpha'}), r_x)\rightarrow X$ of $\gamma'$ with $\gamma_{x}$ for $x=\gamma'(0)$ satisfies $\gamma''(0)=x$ and
\begin{gather}\label{id:rel}
d(\gamma''(s),\gamma''(t))\leq d(\gamma''(s),x)+ d(x, \gamma''(t))= d_S(\gamma''(t))-d_S(\gamma''(s)) \le d(\gamma''(s),\gamma''(t))
\end{gather}
for $s\in (a(X_{\alpha'}),0]$ and $t\in [0,r_x)$.

Thus the inequalities in \eqref{id:rel} are actually equalities. Hence, $\mbox{Im}(\gamma'')\subset R_{d_S}(\gamma''(s))$, the points that are {\it $R_{d_S}$-related} to $\gamma''(s)$. These are exactly the points $y$ that satisfy~\eqref{id:rel} with $\gamma''(t)$ replaced with $y$. But this contradicts the requirement $\overline{X_{\alpha'}}= R_{d_S}(\gamma''(s))$ from the definition of $Q^\dagger$.
\end{proof}

\begin{Remark}[partial converse]
As was pointed out to us by one of the referees
the converse implication in Lemma \ref{lem:ballcondition} holds in the following sense.
If for $x\in S\cap A$ then there exists $p\in X$ that either belongs to $\Omega^c$ or $\Omega^\circ$ such that $B_r(p)$
is either fully contained in $\Omega^c$ or $\Omega^{\circ}$ with $r=d(p,x)$.
\end{Remark}

\begin{Remark}[related literature]
We note that the previous notion of mean curvature under the assumption that $S$ has finite inner curvature, allows to assign to any point $p\in S\cap A^\dagger$ a~number that is the mean curvature of $S$ at $p$. This was useful for proving the Heintze--Karcher inequality in~\cite{kettererHK}.

If one is just interested in lower bounds for the mean curvature, one can adapt a definition of Cavalletti--Mondino~\cite{cm_new}. They define achronal future timelike complete Borel subsets in a~Lorentz length space having {\it forward mean curvature bounded below}. We will not recall their definition for Lorentz length spaces but we give a corresponding definition for $\CDD(K,N)$ metric measure spaces in the appendix of this article and outline how analogs of our results also hold for this notion of lower mean curvature bounds.
\end{Remark}

\section{Proof of inradius bounds and stability (Theorems \ref{T:main} and \ref{T:main2})}

Recall the Jacobian $J_{K,H,N}(r)$ and its maximal interval $r \in (-r_{K,-H,N}, r_{K,H,N})$ of positivity around the origin
defined in \eqref{trig ODE}--\eqref{Jdomain}. To prove our main theorem{s} requires a sort of {below-tangent} implication \eqref{below-tangent} of distorted power concavity \eqref{ineq:conc} from \cite{kettererHK}:

\begin{Lemma}[comparison inequality]\label{lem:key}
Let $h\colon [a,b]\rightarrow [0,\infty)$ be continuous such that $a\leq 0< b$ and
every affine map
$\gamma\colon [0,1]\rightarrow [a,b]$ satisfies
\begin{align}\label{ineq:conc}
h(\gamma_t)^{\frac{1}{N-1}}\geq \sigma_{K/(N-1)}^{(1-t)}(|\dot\gamma|)h(\gamma_0)^{\frac{1}{N-1}}+\sigma_{K/(N-1)}^{(t)}(|\dot\gamma|)h(\gamma_1)^{\frac{1}{N-1}}
\qquad \forall\, t \in [0,1].
\end{align}
Then
\begin{align}\label{below-tangent}
(h(r))^{\frac{1}{N-1}}\leq (h(0))^{\frac{1}{N-1}}\cos_{\frac{K}{N-1}}(r)+ \frac{{\rm d}^+}{{\rm d}s}\Big|_{s=0}(h(s))^{\frac{1}{N-1}} \sin_{\frac{K}{N-1}}(r)
\qquad \forall\, r \in [a,b].
\end{align}
If $h(0)>0$, it follows
$
h(r)h(0)^{-1}\leq J_{K,H,N}(r)
$
where
$H= -\frac{{\rm d}^{+}}{{\rm d}r}\big|_{r=0}\log h(r)$
and in particular $b\leq r_{K,H,N}$.
\end{Lemma}
\begin{proof}If $a<0$, the lemma is exactly the statement of Corollary 4.3 in \cite{kettererHK}.

For $a=0$ we pick $r_n\downarrow 0$. Then, the statement follows since $\frac{{\rm d}^+}{{\rm d}r}h(r)$ is continuous from the left for a semiconcave function $h$.
\end{proof}

\begin{Remark}[reverse parameterization]\label{R:key}
If instead $h\colon [a,b]\rightarrow [0,\infty)$ is continuous and every affine map $\gamma\colon [0,1]\rightarrow [a,b]$ satisfies \eqref{ineq:conc} but
$a < 0 \le b$, then applying Lemma~\ref{lem:key} to $\tilde h(r) := h(-r)$ yields $-a \le r_{K,\tilde H,N}$ with $\tilde H = \frac{{\rm d}^{-}}{{\rm d}r}\big|_{r=0}\log h(r)$.
\end{Remark}

\begin{proof}[Proof of Theorem \ref{T:main}]
Let $(X,d,\m)$ be a $\CDD(K',N)$ space and consider $\Omega\subset X$ satisfying $\CDD_{r}(K,N)$ as assumed in Theorem~\ref{T:main}. Let $u=d_{S}$ be corresponding signed distance function. Let $\{X_\alpha\}_{\alpha\in Q}$ be the decomposition of $\T_{d_S}$ and $\int \m_\alpha \,{\rm d}\mathfrak q(\alpha)$ be the disintegration of $\m$ given by Theorem~\ref{th:1dlocalisation}
and Remark~\ref{R:zero-level selection}.
In Remark~\ref{dagger} we define $Q^\dagger\subset Q$. Recall that~$Q^{\dagger}$ is a subset of~$Q$ with full $\mathfrak q$-measure and for all $\alpha\in Q^{\dagger}$ one has
 ${\dm}_\alpha = h_\alpha \,{\rm d} \mathcal H^1$, $X_{\alpha,e}=\overline{X}_\alpha$ and $h_\alpha$ satisfies
\begin{gather}\label{inequality2}
\big(h_\alpha^{\frac{1}{N-1}}\big)'' + \frac{K}{N-1} h_\alpha^{\frac{1}{N-1}}\leq 0 \qquad \mbox{on} \quad (a(X_\alpha),0)
\end{gather}
in the distributional sense.

Assume $K\in \R$ and $H^-_S\geq \kkapp(N-1)$ $\m_S$-a.e.\ Recall that
\begin{gather*}
H^-_S(\gamma_\alpha(0)) =\frac{{\rm d}^-}{{\rm d}r}\Big|_{r=0}\log h_\alpha(r).
\end{gather*}
In particular, $h_\alpha(0)>0$ for $\mathfrak q$-a.e.\ $\alpha\in Q^{\dagger}$, since $h_\alpha(0)=0$ yields $\frac{{\rm d}^-}{{\rm d}r}\log h_\alpha(r)=-\infty< \kkapp(N-1)$ by Definition~\ref{def:meancurvature}.
Now~\eqref{inequality2} implies $h_\alpha$ can be represented by a continuous function on $[a(X_\alpha),0]$ which satisfies the hypotheses of Remark~\ref{R:key}.
That remark then asserts
$-a(X_\alpha) \leq r_{K,\kkapp(N-1),N}$ for any $\alpha\in Q^{\dagger}$.

Since $(\m_\alpha|_{\Omega})_{\alpha\in Q^{\dagger}}$ is a disintegration of $\m|_\Omega$, and since $\m_\alpha|_{\Omega}$ is supported on $\mbox{Im}(\gamma_\alpha|_{[a(X_\alpha), 0]})$ where $\gamma_\alpha\colon [a(X_\alpha, 0]\rightarrow \overline{\Omega}$ is a geodesic, it follows that for $\m$-almost every $x\in \Omega$ there exists $\alpha\in Q^\dagger$ and $t\in [a(X_\alpha),0]$ such that $x=\gamma_\alpha(t)$. For such $x$ it follows that $d_{\Omega^c}(x)= -t \leq -a(X_\alpha)\leq r_{K,\kkapp(N-1),N}$. Hence $d_{\Omega^c}\leq r_{K,\kkapp(N-1),N}$ for $\m$-almost everywhere in $\Omega$.

By continuity of $d_{\Omega^c}$ and $X=\supp \m$ it follows that $d_{\Omega^c}(x)\leq r_{K,\kkapp(N-1),N}$ for all $x\in \Omega$.
In particular, the inscribed radius satisfies $\inrad\Omega \leq r_{K,\kkapp(N-1),N}$.
\end{proof}

\begin{proof}[Proof of Theorem \ref{T:main2}]
Assume $K\geq \bar K -\delta$, $H^-_S\geq \bar H -\delta$ $\m_S$-a.e.\ and $N\leq \bar{N}+\delta$. Since $\delta>0$ and $X \in \CDD(K,N)$ imply
$X \in \CDD(\bar K -\delta, \bar N +\delta)$, Theorem \ref{T:main} yields
\[ \inrad \Omega \leq r_{\bar K-\delta, \bar H-\delta, \bar N+\delta}.\]

Now for any $\epsilon>0$ there exists $\delta>0$ such that $r_{\bar K-\delta, \bar H-\delta, \bar N+\delta }\leq r_{\bar K, \bar H, \bar N}+\epsilon$ since the function
$s_{\frac{K}{N-1},\frac{H}{N-1}}(r)=\cos_{\frac{K}{N-1}}(r) - \frac{H}{N-1}\sin_{\frac{K}{N-1}}(r)$ whose first positive zero defines $r_{K,H,N}$ is continuously differentiable with respect to $K$, $H$, $N$ and $r$, and its derivative is non-zero at $r=r_{\bar K,\bar H,\bar N}$; the implicit function theorem then gives
continuous differentiability of $r_{\bar K,\bar H,\bar N}$ with respect to its parameters near any $\big(\frac{\bar K}{N-1},\frac{\bar H}{N-1}\big)$ satisfying the ball condition.
If the ball condition is not satisfied, then $r_{\bar H,\bar K,\bar N}=\infty$ and the theorem holds trivially.
\end{proof}

\section{Rigidity}\label{sec:rig}
\subsection{The Riemannian curvature-dimension condition}

We recall briefly the \textit{Riemannian curvature-dimension condition} that is a strengthening of the $\CDD(K,N)$ condition and the result of the combined efforts by several authors \cite{agmr,agsriemannian, amsnonlinear, cavmil, erbarkuwadasturm, giglistructure}.

The {\it Cheeger energy} $\Ch\colon L^2(\m)\rightarrow [0,\infty]$ of a metric measure space $(X,d,\m)$ is defined as
\begin{gather}\label{Cheeger energy}
2\Ch(f):=\liminf_{\lip(X)\ni u_n \overset{L^2}{\rightarrow} f} \int (\lip u_n)^2 \dm,
\end{gather}
where $\lip(X)$ is the space of Lipschitz functions on $(X,d,\m)$ and $\lip u(x):=\limsup\limits_{y\rightarrow x}\frac{|u(x)-u(y)|}{d(x,y)}$ is the local slope of $u\in \lip(X)$.
The $L^2$-Sobolev space is defined as
$W^{{1,2}}(X)=\{f\in L^2(\m)\colon$ $\Ch(f)<\infty\}$
and equipped with the norm $\|f\|^2_{}:={\|f\|_{\scriptscriptstyle{L^2(\m)}}^2+ 2\Ch(f)}$~\cite{agslipschitz, agsheat}.

For $u\in W^{1,2}(X)$ the Cheeger energy can be written as
\begin{gather*}
2\Ch(u)=\int_X|\nabla u|^2 \dm
\end{gather*}
for a measurable density $|\nabla u|\colon X\rightarrow [0,\infty)$ that is identified as the {\it minimal weak upper gradient} of $u$. For more details about the minimal weak upper gradients and its characterizations we refer to \cite{agslipschitz, cheegerlipschitz}.

\begin{Definition}[Riemannian curvature-dimension condition]\label{def:RCD}
A metric measure space $(X,d,\m)$ satisfies the \textit{Riemannian curvature-dimension condition} $\RCRD(K,N)$ if $(X,d,\m)$ satisfies the condition $\CDD(K,N)$ and $W^{1,2}(X)$ is a~Hilbert space,
meaning the $2$-homogeneous Cheeger energy~\eqref{Cheeger energy} satisfies the parallelogram law:
\[
\Ch(f+g) + \Ch(f-g) = 2 \Ch(f) + 2 \Ch(g).
\]
\end{Definition}

When $(X,d,\m)$ is an $\RCRD$ space, one can introduce a symmetric bilinear form $\langle \cdot, \cdot \rangle$ on the Sobolev space $W^{1,2}(X)$ with values in $L^1(\m)$ via
\begin{gather*}
(f,g)\in W^{1,2}(X)\times W^{1,2}(X)\mapsto \langle \nabla f,\nabla g\rangle:=
\frac{1}{4}|\nabla (f+g)|^2 - \frac{1}{4}|\nabla (f-g)|^2
\in L^1(\m).
\end{gather*}

\subsection{Volume cone implies metric cone}
The following theorem by Gigli and De Philippis will be crucial in the proof of the rigidity result.

\begin{Theorem}[{volume cone implies metric cone \cite[Theorem 4.1]{DGi}}]\label{thm:gp}
Let $K\!\in\! \{-(N-1),0, N-1\}$, $N\in [1,\infty)$ and $(X,d,\m)$ an $\RCRD(K,N)$ space with $\supp\m=X$. Assume there exists $o\in X$ and $R>r>0$ such that
\begin{gather}\label{volumecone}
\m(B_R(o))= \frac{{\displaystyle \int_0^R\big( \sin_{\frac{K}{N-1}} u\big)^{N-1}{\rm d}u}}{{\displaystyle \int_0^r\big( \sin_{\frac{K}{N-1}}v \big)^{N-1}{\rm d}v}} \m(B_r(o)).
\end{gather}
Then exactly one of the following three cases holds:
\begin{itemize}\itemsep=0pt
\item[$(1)$] If $\partial B_{R/2}(o)$ contains only one point, then $X$ is isometric to $[0,\diam_X]$ $($or $[0,\infty)$ if $X$ is unbounded$)$ with an isometry that sends $o$ to $0$ either way. The measure $\m|_{B_R(o)}$ is proportional to $\big(\sin_{\frac{K}{N-1}}^{N-1} x\big){\rm d}x$.

\item[$(2)$] If $\partial B_{R/2}(o)$ contains exactly two points then $X$ is a $1$-dimensional Riemannian manifold, possibly with boundary, and there exists a bijective, locally distance preserving map from~$B_{R}(o)$ to~$(-R,R)$ that sends $o$ to $0$ under which the measure $\m|_{B_R(o)}$ becomes proportional to
$\big(\sin_{\frac{K}{N-1}}^{N-1}|x|\big) {\rm d}x$.

\item[$(3)$] If $\partial B_{R/2}(o)$ contains more than two points then $N\!\geq \!2$ and there exists an $\RCRD(N{-}2, N{-}1)$ space $Z$ with $\diam_Z\leq \pi$ and a local isometry ${U}\colon B_R(o) \rightarrow [0,R)\times_{\sin_{\frac{K}{N-1}}}^{N} Z$ sending~$o$ to~$0$ that is also a measure preserving bijection.
\end{itemize}
\end{Theorem}

\begin{Remark}[excluding the middle case]\label{rem:refined0}
In the second case the conclusion also implies that $N=1$: otherwise $X$ is locally isomorphic to $(-R,R)$ equipped with a measure proportional to $\big(\sin_{{K}/{N-1}}^{N-1}|x|\big){\rm d}x$. But for $N>1$ this space does not satisfy the $\CDD$ condition because the density of the reference measure vanishes at $0$. This means the density is not semi-concave which is a necessary condition for the measure on a $1D$ space to satisfy the $\CDD$ condition.
\end{Remark}

\begin{Remark}\label{rem:refined}
In the proof of Theorem \ref{thm:gp} Gigli and De Philippis show that the map $U$ has an inverse ${V}\colon B_R(0) \rightarrow B_R(o)$ that is also a local isometry.
\end{Remark}

\subsection{Distributional Laplacian and strong maximum principle}We recall the notion of {the} distributional Laplacian for $\RCRD$ spaces (cf.~\cite{cav-mon-lapl-18, giglistructure}).

Let $(X,d,\m)$ be an $\RCRD$ {space,
and $\lip_c(\Omega)$ denote the set of Lipschitz functions compactly supported in an open subset $\Omega \subset X$.}
A {\em Radon functional} over $\Omega$ is a linear functional $T: \lip_c(\Omega)\rightarrow \R$ such that for every compact subset $W$ in $\Omega$ there exists a constant $C_W\geq 0$ such that
\begin{gather}\label{Radon}
|T(f)|\leq {C_W} \max_W|f| \qquad \forall\, f\in \lip_c(\Omega) \quad \mbox{with} \ \supp f\subset W.
\end{gather}
One says $T$ is non-negative if $T(f)\geq 0$ for all $f\in \lip_c(\Omega)$ satisfying $f\geq 0$.

The classical Riesz--Markov--Kakutani representation theorem says that for every non-nega\-ti\-ve Radon functional $T$ from~\eqref{Radon} there exists a non-negative Radon measure $\mu_T$ such that $T(f)=\int f \,{\rm d}\mu_T$ for all $f\in \lip_c(\Omega)$.

\begin{Definition}[nonsmooth Laplacian]Let $\Omega\subset X$ be an open subset and let $u\in \lip(X)$. One says $u$ is in domain of the distributional Laplacian on $\Omega$ provided there exists a Radon functional $T$ over $\Omega$ such that
\begin{gather*}
T(f) =\int \langle \nabla u, \nabla f\rangle \dm \qquad \forall\, f \in \lip_c(\Omega).
\end{gather*}
In this case we write $u\in D({\bf \Delta}, \Omega)$. If $T$ is represented as a measure $\mu_T$, one writes $\mu_T\in {\bf \Delta}u|_{\Omega}$, and if there is only one such measure $\mu_T$ by abuse of notation we will identify $\mu_T$ with $T$ and write $\mu_T={\bf \Delta} u|_{\Omega}$.
\end{Definition}
We also recall that $u\in W_{\rm loc}^{1,2}(\Omega)$ for an open set $\Omega\subset X$ if and only if for any Lipschitz function $\phi$ with compact support in $\Omega$ we have $\phi \cdot u\in W^{1,2}(X)$.
In particular, if $u\in \lip(X)$ then $u\in W_{\rm loc}^{1,2}(\Omega)$.

\begin{Remark}[{locality and linearity}]\quad
\begin{itemize}\itemsep=0pt
\item[(i)] If $u\in D({\bf \Delta}, \Omega)$ and $\Omega'$ is open in $X$ with $\Omega'\subset \Omega$, then $u\in D({\bf \Delta},\Omega')$ and for $\mu\in {\bf \Delta}u|_{\Omega}$ it follows that $\mu|_{\Omega'}\in {\bf \Delta}u|_{\Omega'}$.
\item[(ii)] If $u,v\in D({\bf \Delta}, \Omega)$, then $u+v\in D({\bf \Delta},\Omega)$ and for $\mu_u\in {\bf \Delta}u|_{\Omega}$ and $\mu_v\in {\bf \Delta}v_{\Omega}$ it follows that $\mu_u+\mu_v\in {\bf \Delta}(u+v)|_{\Omega}$.
\end{itemize}
\end{Remark}

Recall that $u\in W^{1,2}(\Omega)$ is {\it sub-harmonic} if
\begin{gather*}
\int_{\Omega}\! |\nabla u|^2 \dm\leq \!\int_{\Omega}\! |\nabla (u+g)|^2 \dm \!\!\!\qquad \forall\, g\in W^{1,2}(\Omega) \ \mbox{with} \ g \leq 0 \  \mbox{compactly supported in $\Omega^\circ$.}
\end{gather*}
One says $u$ is {\it super-harmonic} if $-u$ is sub-harmonic, and $u$ is {\it harmonic} if it is both sub- and super-harmonic. The following can be found in \cite[Theorem~4.3]{giglimondino}.

\begin{Theorem}[characterizing super-harmonicity]\label{thm:sh}
Let $X$ be an $\RCRD(K,N)$ space with $K\in \R$ and $N\in [1,\infty)$, let $\Omega \subset X$ be open and $u\in W_{\rm loc}^{1,2}(\Omega)$.
Then $u$ is {super-harmonic} if and only if ${u} \in D({\bf \Delta}, \Omega)$ and there exists $\mu\in {\bf \Delta} u|_{\Omega}$ such that $\mu \leq 0$.
\end{Theorem}

The following is \cite[Theorem 9.13]{bjoern} (see also~\cite{gigli_rigoni}):
\begin{Theorem}[strong maximum principle]\label{thm:mp}
Let $X$ be an $\RCRD(K,N)$ space with $K\in \R$ and $N\in [1,\infty)$, let $\Omega\subset X$ be a connected open set with compact closure and let $u\in W_{\rm loc}^{1,2}(\Omega)\cap C(\Omega)$ be sub-harmonic.
If there exists $x_0\in \Omega$ such that $u(x_0)=\max_{\bar \Omega} u$ then $u$ is constant.
\end{Theorem}

Let us recall another result of Cavalletti--Mondino:

\begin{Theorem}[{Laplacian of a signed distance, \cite[Corollary 4.16]{cav-mon-lapl-18}}]\label{thm:cm}
Let $(X,d,\m)$ be a $\CDD(K,N)$ space, and $\Omega$ and $S=\partial \Omega$ as above. Then
$d_S\in D({\bf \Delta}, X\backslash S)$, and one element of ${\bf \Delta}d_S|_{X\backslash S}$ that we also denote with ${\bf \Delta }d_S|_{X\backslash S}$ is the Radon functional on $X\backslash S$ given by the representation formula
\begin{gather*}
({\bf \Delta} d_S)|_{X\backslash S}= (\log h_{\alpha})'\m|_{X\backslash S}+\int_Q ( h_{\alpha}\delta_{a(X_{\alpha})\cap \{d_S < 0\}} - h_{\alpha}\delta_{b(X_{\alpha})\cap \{d_S >0\}} ) \,{\rm d}\mathfrak q(\alpha).
\end{gather*}
We note that the Radon functional ${\bf \Delta} d_S|_{X\backslash S}$ can be represented as the difference of two measures $[{\bf \Delta} d_S]^+$ and $[{\bf \Delta} d_S|_{X\backslash S}]^-$ such that
\begin{align*}
[{\bf \Delta} d_S|_{X\backslash S}]^+_{\rm reg} - [{\bf \Delta} d_S|_{X\backslash S}]^-_{\rm reg} = (\log h_{\alpha})' \quad \m\mbox{-a.e.},
\end{align*}
where $[{\bf \Delta} d_S|_{X\backslash S}]^{\pm}_{\rm reg}$ denotes the $\m$-absolutely continuous part in the Lebesgue decomposition of $[{\bf \Delta} d_S|_{X\backslash S}]^{\pm}$. In particular, $-(\log h_{\alpha})'$ coincides with a measurable function $\m$-a.e.
\end{Theorem}

To prove the rigidity asserted in Theorem~\ref{T:main3}, we need one more lemma:

\begin{Lemma}[Riccati comparison]\label{lem:riccati}
Let $u\colon [0,b]\rightarrow \R$ be non-negative and continuous such that $u''+\kappa u\leq 0$ in the distributional sense, $u(0)=1$ and $u'(0)\leq d$.
Let $v\colon [0,\bar b]\rightarrow \R$ be the maximal non-negative solution of $v''+\kappa v = 0$ with $v(0)=1$ and $v'(0)=d$. That is,
$v=s_{\kappa,-d}$ from~\eqref{equ:ks}.
Then $\bar b\geq b$ and $\frac{{\rm d}^+}{{\rm d}t} \log u\leq (\log v)' $ on $[0,b)$.
\end{Lemma}

\begin{proof} Note that $v(r)=\cos_\kappa(r) + d\sin_\kappa(r)$ and $\bar b=r_{\kappa(N-1), -d(N-1),N}$. Then Lemma~\ref{lem:key} already yields that $b\leq \bar b$ and $u\leq v$ on $[0,b]$.
Therefore, without loss of generality we restrict $v$ to $[0,b]$.

We pick $\varphi \in C^{2}_c(\R)$ compactly supported in $(-1,1)$
with $\int \varphi \,{\rm d} \mathcal L^1 =1$ and define $\varphi_\epsilon(x)=\epsilon\varphi\big(\frac{x}{\epsilon}\big)$. Let $\bar\epsilon>{0}$ and $\epsilon\in (0,\bar\epsilon)$, and let $u_\epsilon=\int \varphi_\epsilon(t)u(s-t) \,{\rm d}t$ be the mollification of $u$ by $\varphi_\epsilon$. One can check that $u_\epsilon$ is well-defined on $[\bar\epsilon, b- \bar \epsilon]$ and $u_\epsilon \in C^2([\bar\epsilon, b-\bar\epsilon])$ satisfies
\begin{gather*}
u_\epsilon'' +(\kappa+\delta) u_\epsilon\leq 0
\end{gather*}
in the classical sense with $\delta=\delta(\epsilon)\rightarrow 0$ for $\epsilon \rightarrow 0$. Since $u$ is continuous, $u_\epsilon(t)\rightarrow u(t)$ for all $t\in [\bar\epsilon, b-\bar\epsilon]$. Moreover, $u_\epsilon'(t)\rightarrow u'(t)$ for every $t\in [\bar\epsilon, b-\bar\epsilon]$ where $u$ is differentiable.

Let $v_\epsilon\colon [0,\bar b_\epsilon]\rightarrow [0,\infty)$ be the maximal positive solution of $v_\epsilon''+(\kappa+\delta(\epsilon))v_\epsilon=0$ with $v_\epsilon(0)=1$ and $v_\epsilon'(0)=d$.
Since $\delta(\epsilon)\rightarrow 0$ for $\epsilon\rightarrow 0$ we have $\bar b_\epsilon \rightarrow \bar b$, $v_\epsilon \rightarrow v$ and $v'_\epsilon \rightarrow v$ pointwise on $[0,\bar b]$ if $\epsilon \rightarrow 0$.

We pick $\bar \epsilon \in (0,b)$ and $t\in [\bar \epsilon , b-\bar\epsilon]$ where $u$ is differentiable. Then
\begin{align*}
0&\geq \int_{\bar\epsilon}^t [ v_\epsilon (u_\epsilon''+ (\kappa+\delta)u_\epsilon) - u_\epsilon (v_\epsilon''+ (\kappa+ \delta) v_\epsilon) ]\,{\rm d}\mathcal L^1
 = \int_{\bar\epsilon}^t \{[ v_\epsilon u_\epsilon' ]' - [u_\epsilon v_\epsilon' ]' \}\,{\rm d}\mathcal L^1 \\
&= v_\epsilon(t) u_\epsilon'(t) - u_\epsilon(t) v_\epsilon'(t) + u_\epsilon (\bar\epsilon) v_\epsilon'(\bar\epsilon) - v_\epsilon(\bar\epsilon) u_\epsilon'(\bar\epsilon)
\\
&
\rightarrow v(t) u'(t) - u(t) v'(t) + u (\bar\epsilon) v'(\bar\epsilon) - v(\bar\epsilon) u'(\bar\epsilon).
\end{align*}
Since $u$ is semiconcave and continuous on $[0, b]$, the right derivative $\frac{{\rm d}^+}{{\rm d}t}u\colon [0,b]\rightarrow \R\cup\{\infty\}$ is continuous from the right.
Hence, for $\bar\epsilon \downarrow 0$ and any $t\in (0,b)$ it follows
\begin{gather*}
0\geq v(t) \frac{{\rm d}^+}{{\rm d}t}u(t) - u(t) v'(t) + u(0) v'(0) - v(0) \frac{{\rm d}^+}{{\rm d}t}u(0)\geq v(t) \frac{{\rm d}^+}{{\rm d}t}u(t) - u(t) v'(t).
\end{gather*}
 Hence $\frac{{\rm d}^+}{{\rm d}t}\log u= \frac{\frac{{\rm d}^+}{{\rm d}t}u}{u}\leq \frac{v'}{v}=(\log v)'$ as desired.
\end{proof}

We obtain the following improved Laplace comparison statement for distance functions in $\CDD(K,N)$ spaces that may be of interest in its own right. In the smooth context the result was obtained by Kasue \cite[Corollary~2.44]{kasue-laplace} for Riemannian manifolds and by Sakurai \cite[Lemma~3.3]{sakurai} for weighted Riemannian manifolds. For weighted Finsler manifolds that satisfy a lower Bakry--\'Emery Ricci curvature bound in the sense of~\cite{ohtafinsler1} the result seems to be new.
\begin{Corollary}[improved Laplace comparison]\label{cor:laplace}
Let $(X,d,\m)$ be an essentially nonbranching $\CDD(K',N)$ space with $K'\in \R$, $N\in (1,\infty)$ and $\supp \m=X$. For $K, \kkapp\in \R$, let $\Omega\subset X$ be closed with $\Omega \neq X$, $\m(\Omega)>0$ and $\m(\partial \Omega)=0$ such that $\Omega$ satisfies $\CDD_r(K,N)$ and $\partial \Omega=S$ has finite inner curvature. Assume the inner mean curvature $H_S^-$ satisfies $H_S^-\geq (N-1)\kkapp$ $\m_S$-almost everywhere. Then
\begin{gather*}
({\bf \Delta} d_{\Omega^c})|_{\Omega^\circ}\leq (N-1) \frac{s'_{\frac{K}{N-1}, \kkapp}(d_{\Omega^c})}{s_{\frac{K}{N-1}, \kkapp}(d_{\Omega^c})}\m|_{\Omega^{\circ}},
\end{gather*}
where $s_{\kappa,\lambda}$ was defined in equation~\eqref{equ:ks}.
\end{Corollary}
\begin{proof}
Let $u=d_{S}$ be the signed distance function of $S$. Let $\{X_\alpha\}_{\alpha\in Q}$ be the decomposition of~$\T_u$ and $\int \m_\alpha \,{\rm d}\mathfrak q(\alpha)$ be the disintegration of $\m$ given by Theorem~\ref{T:CM disintegration}
and Remark~\ref{R:zero-level selection}. Recall that $m_\alpha=h_\alpha \mathcal H^1$ for $\mathfrak q$-a.e.\ $\alpha\in Q$.
We consider $Q^\dagger\subset Q$ that has full $\mathfrak q$-measure as defined in Remark~\ref{dagger}. For every $\alpha\in Q^{\dagger}$ we have that
$\m_\alpha = h_\alpha \mathcal H^1$,
$X_{\alpha,e}=\overline{X}_\alpha$ and $h_\alpha$ is continuous on $[a(X_\alpha),0]$ by Remark \ref{rem:extden} and satisfies
\begin{gather}\label{inequality}
\big(h_\alpha^{\frac{1}{N-1}}\big)'' + \frac{K}{N-1} h_\alpha^{\frac{1}{N-1}}\leq 0 \qquad \mbox{on} \ (a(X_\alpha),0) \qquad \forall\, \alpha\in Q^{\dagger},
\end{gather}
in the distributional sense. Note that we have the constant $K$ because of Remark~\ref{rem_mer}. As usual we write $h_\alpha=h_\alpha\circ \gamma_\alpha$. We also have the properties of $h_\alpha$ as discussed in Remark~\ref{rem:kuconcave}. The function $r\in [0,-a(X_\alpha)]\mapsto \tilde h_\alpha(r) := h_\alpha(-r)$ is also continuous and~\eqref{inequality} is still holds on $(0,-a(X_\alpha))$.
Recall that by the lower mean curvature bound the set of $\alpha$'s in $Q$ with $h_\alpha(0)=0$ has $\mathfrak q$-measure $0$. Hence $h_\alpha(r)>0$ for $\mathfrak q$-a.e.~$\alpha$.

Therefore, for $\mathfrak q$-a.e.\ $\alpha\in Q^{\dagger}$ we have that
$\big[\tilde h_\alpha(r)/\tilde h_\alpha(0)\big]^\frac{1}{N-1}=:u(r)$ satisfies $u''(r)+\frac{K}{N-1} u(r)\leq 0$ in the distributional sense with $u(0)=1$.
Moreover, we have
\begin{gather*}
\kkapp(N-1)\leq H^-_S(\gamma_\alpha(0)) =\frac{{\rm d}^-}{{\rm d}r}\Big|_{r=0}\log h_\alpha(r) =- \frac{{\rm d}^+}{{\rm d}r}\Big|_{r=0}\log \tilde h_\alpha(r)
\end{gather*}
and therefore
\begin{gather*}
 \frac{{\rm d}^+}{{\rm d}r}\Big|_{r=0} u(r) = \frac{{\rm d}^+}{{\rm d}r}\Big|_{r=0} \left[\frac{\tilde h_\alpha(r)}{\tilde h_\alpha(0)}\right]^{\frac{1}{N-1}}\leq - \kkapp.
\end{gather*}

By Theorem \ref{thm:cm}, $d_S\in D({\bf \Delta}, X\backslash S)$ and
\begin{gather*}
({\bf \Delta} d_S)|_{X\backslash S}= (\log h_\alpha)' \m|_{X\backslash S}+\int_Q ( h_{\alpha}\delta_{a(X_{\alpha})\cap \{d_S{\color{black}<}0\}} - h_{\alpha}\delta_{b(X_{\alpha})\cap \{d_S{\color{black}>}0\}} ) \,{\rm d}\mathfrak q(\alpha).
\end{gather*}
Recall that $-d_S|_{\Omega^\circ}=d_{\Omega^c}|_{\Omega^\circ}$ and by locality of the distributional Laplacian $(({\bf \Delta}d_{S})|_{X\backslash S})|_{\Omega^\circ}= {\bf \Delta} (d_{S}|_{\Omega^\circ})= {\bf \Delta} (-d_{\Omega^{c}}|_{\Omega^\circ})$. Hence
\begin{gather*}
{\bf \Delta}(- d_{\Omega^{c}}|_{\Omega^\circ})= (\log h_\alpha)' \m|_{\Omega^\circ}
+
\int_Q {(h_{\alpha}\delta_{a(X_{\alpha})\cap \{d_S < 0\}} -
h_{\alpha}\delta_{b(X_{\alpha})\cap \{d_S >0\}} } \,{\rm d}\mathfrak q(\alpha).
\end{gather*}
Any $\gamma_{\alpha}$ for $\alpha\in Q$ that starts inside of $\Omega^\circ$ satisfies $\gamma_\alpha(b(X_\alpha))\in \overline{\Omega^c}$. Hence $\int_Q h_{\alpha}\delta_{b(X_{\alpha})\cap\Omega^{\circ}} \,{\rm d}\mathfrak q(\alpha)\allowbreak =0$.
Recall that $-(\log h_\alpha)'(r)=(\log \tilde h_\alpha)'(-r)$. It follows that
\begin{gather*}
({\bf \Delta} d_{\Omega^c} )|_{\Omega^\circ}= -{\bf \Delta}(- d_{\Omega^c})|_{\Omega^\circ}
 \leq - (\log h_\alpha)' \m|_{\Omega^\circ}= (\log \tilde h_\alpha)' \m|_{\Omega^\circ}.
\end{gather*}
In the first equality it seems as one would use linearity of the Laplacian to pull the minus sign in front of ${\bf \Delta}$. But if one examines the proof of Theorem \ref{thm:cm}, one can easily observe that this is possible also in $\CDD$ (or even $\MCP$ context) because replacing $d_S$ with~$-d_S$ just results in reversing the parametrization of the geodesics $\gamma_\alpha$.

Now the corollary follows immediately from the curvature bounds and the Riccati comparison lemma (Lemma~\ref{lem:riccati}).
\end{proof}

\subsection{Proof of rigidity (Theorem \ref{T:main3})}

1.\ We assume that $K \in \{N-1,0,-(N-1)\}$.

Now $d_S|_{{\Omega^\circ}}=-d_{\Omega^c}|_{{\Omega^\circ}}\in D({\bf \Delta}, {\Omega^\circ})$ and by Corollary \ref{cor:laplace} we have
\begin{align}\label{yyyy}
\frac{1}{N-1}{\bf \Delta} d_{\Omega^c}|_{{\Omega^\circ}}\leq \frac{s_{\frac{K}{N-1}, \kkapp}'(d_{\Omega^c})}{s_{\frac{K}{N-1}, \kkapp}(d_{\Omega^c})} \m|_{\Omega^\circ}=
\frac{- \frac{K}{N-1}\sin_{\frac{K}{N-1}} d_{\Omega^c} - \kkapp \cos_{\frac{K}{N-1}} d_{\Omega^c}}{\cos_{\frac{K}{N-1}} d_{\Omega^{c}}- \kkapp\sin_{\frac{K}{N-1}} d_{\Omega^c}} \m|_{{\Omega^\circ}}.
\end{align}
Assume equality holds in the inradius bound \eqref{inradbound}, meaning $d_{\Omega^c}(p)=r_{K,\kkapp(N-1),N} <\infty$ for some $p\in {\Omega^\circ}$ by Remark~\ref{R:Heine-Borel}. In particular, there exists a geodesic $\gamma^*\colon [a,b]\rightarrow \Omega$ of length ${L}(\gamma^*)=d(\gamma^*(a),\gamma^*(b))={r_{K,\kkapp(N-1),N}}$
 such that $\gamma^*(a)=p$ and $\gamma^*(b)\in S:= \partial \Omega$. Moreover $B_{r_{K,\kkapp(N-1),N}}(p)\subset {\Omega^\circ}$ and
\begin{align}\label{xxxx}
\frac{1}{N-1}{\bf \Delta} d_p|_{{\Omega^\circ\backslash\{p\}}}\leq\frac{\cos_{\frac{K}{N-1}} d_p}{\sin_{\frac{K}{N-1}} d_p} \m|_{{\Omega^\circ\backslash\{p\}}}
\end{align}
by the Laplace comparison theorem \cite[Corollary 5.15]{giglistructure} (heuristically the limit $\kkapp \to \infty$ in inequality \eqref{yyyy};
see also \cite[Theorem~1.1]{cav-mon-lapl-18}).

We will add the previous inequalities \eqref{yyyy} and \eqref{xxxx}.
We first note that \begin{gather*}
\sin_{\frac{K}{N-1}} (d_p) \big[\cos_{\frac{K}{N-1}} d_{\Omega^c} - \kkapp\sin_{\frac{K}{N-1}} d_{\Omega^c}\big]=
\sin_{\frac{K}{N-1}} (d_p)\big[ s_{\frac{K}{N-1}, \kkapp}(d_{\Omega^c})\big]> 0 \qquad \mbox{$\m$-a.e.\ on }\Omega^\circ.
\end{gather*}
This is true because on the one hand $J_{K,\kkapp(N-1),N}(d_{\Omega^c})=(s_{{K}/{(N-1)},\kkapp}(d_{\Omega^c}))^{N-1}>0$ $\m$-a.e.\ on~$\Omega^\circ$. The latter is easy to see
using the disintegration induced by $d_{\Omega^c}$ on $\Omega^\circ$ and Theorem~\ref{T:main}. For the other factor recall when $K=N-1>0$ that $d_p\leq \pi$ for any $p\in X$ by the Bonnet--Myers diameter estimate (e.g.,~\cite{stugeo2}) with at most one point $q\neq p$ where $d_p(q)=\pi$~\cite{ohtmea}.

Adding the inequalities \eqref{yyyy} and \eqref{xxxx} and using the linearity of the Laplace operator yields
\begin{gather*}
 \big(\sin_{\frac{K}{N-1}} (d_p)\big[ s_{\frac{K}{N-1}, \kkapp}(d_{\Omega^c})\big]\big)\big|_{\Omega^\circ\backslash \{p\}}\frac{{\bf \Delta} (d_p+d_{\Omega^c})}{N-1}\Big|_{{\Omega^\circ\backslash\{p\}}} \\
\qquad {} \leq
\left(\big[\cos_{\frac{K}{N-1}}d_p\big]\big[ \cos_{\frac{K}{N-1}} d_{\Omega^c}\big]-\big[ \sin_{\frac{K}{N-1}} d_p \big]\frac{K}{N-1} \big[\sin_{\frac{K}{N-1}} d_{\Omega^c} \big]\right)\m|_{{\Omega^\circ\backslash\{p\}}}
\\
\qquad\quad{} - \kkapp \big(\big[\sin_{\frac{K}{N-1}} d_p\big]\big[ \cos_{\frac{K}{N-1}} d_{\Omega^c}\big]+\big[\cos_{\frac{K}{N-1}} d_p\big] \big[\sin_{\frac{K}{N-1}} d_{\Omega^c} \big]\big) \m|_{{\Omega^\circ\backslash\{p\}}}\\
\qquad{} = \big( {\cos_{\frac{K}{N-1}} (d_p +d_{\Omega^c})-\kkapp \sin_{\frac{K}{N-1}} (d_p+ d_{\Omega^c})}\big)\m|_{{\Omega^\circ\backslash\{p\}}}\\
\qquad{} =
{s_{\frac{K}{N-1},\kkapp}(d_p+d_{\Omega^c})}\m|_{\Omega^\circ\backslash\{p\}}
\leq 0.
\end{gather*}
The last inequality follows from $r_{K,\kkapp(N-1),N}\leq d_p+d_{\Omega^c}\leq 2 r_{K,\kkapp(N-1),N}$ by the triangle inequality, the definition of $r_{K,\kkapp(N-1),N}$,
the period $<2r_{K,\kkapp(N-1),N}$ of the sinusoid $s_{\frac{K}{N-1},\kkapp}$, and Theorem~\ref{T:main}.

Hence $d_p+ d_{\Omega^c}$ on ${\Omega^\circ\backslash \{p\}}$ is a super-harmonic function that attains its minimum $r_{K,\kkapp(N-1),N}$ inside ${\Omega^\circ}\backslash \{p\}$ along the geodesic $\gamma^*$.
Therefore, $d_p+d_{\Omega^c}=r_{K,\kkapp(N-1),N}$ on the connected component $\Omega^*$ of $\gamma^*$ in ${\Omega^\circ}\backslash \{p\}$ by the strong maximum principle (Theorem~\ref{thm:mp}).

In particular, it follows that ${\Omega^*}\subset B_{r_{K,\kkapp(N-1),N}}(p)$ and $a(X_\alpha)=-r_{K,\kkapp(N-1),N}$ for $\mathfrak q$-a.e.\ $\alpha\in Q$ such that $\mbox{Im}\gamma_\alpha$ intersects with $\Omega^*$ (and therefore $\gamma_\alpha(I_\alpha)\subset \Omega^*$). Moreover, both inequalities \eqref{yyyy} and \eqref{xxxx} are saturated throughout ${\Omega^*}$:
\begin{gather*}
{\bf \Delta} d_{\Omega^c}|_{\Omega^*} =
(N-1)\frac{- \frac{K}{N-1}\sin_{\frac{K}{N-1}} d_{\Omega^c} - \kkapp \cos_{\frac{K}{N-1}} d_{\Omega^c}}{\cos_{\frac{K}{N-1}} d_{\Omega^{c}}- \kkapp\sin_{\frac{K}{N-1}} d_{\Omega^c}} \m|_{{\Omega^*}}.
\end{gather*}
 By Theorem \ref{thm:cm},
\begin{align*}
({\bf \Delta}d_p)\circ \gamma_\alpha(r) = -( {\bf \Delta}d_{\Omega^c})\circ \gamma_\alpha(r)= (\log h_\alpha)'\circ \gamma_\alpha(r)
\end{align*}for $r\in [-r_{K,\kkapp(N-1),N},0)$ and $\mathfrak q$-a.e.\ $\alpha\in Q$.
Recall that $(\log h_\alpha)'\circ \gamma_\alpha(r)$ is in fact given by $(\log h_\alpha)'(r)$ for the density $h_\alpha$ of $\m_\alpha$ {with respect to} $\mathcal L^1$ on $[a(X_\alpha),b(X_\alpha)]$.

Solving the resulting ODE for $h_\alpha$ yields
\[
h_\alpha(r)=h_\alpha(0) J_{K,\kkapp(N-1),N}(-r)\qquad \mbox{for $r\in [-r_{K,\kkapp(N-1),N},0)$ and $\mathfrak q$-a.e.\ $\alpha\in Q$}.
\]
Proportionality of $s_{\frac{K}{N-1},\kkapp}(r_{K,\kkapp(N-1),N}-r)$ to $\sin_{\frac{K}{N-1}}(r)$ therefore provides $\lambda>0$ such that
\begin{gather*}
\m(B_R(p)\cap \Omega^*)=\lambda \int_0^R (\sin_{\frac{K}{N-1}}r)^{N-1} \,{\rm d}r \qquad \forall\, R\in [0,r_{K,\kkapp(N-1),N}].
\end{gather*}
Hence $B_{r_{K,\kkapp(N-1),N}}(p)\cap \Omega^*= \Omega^*$ is a volume cone in the sense of~\eqref{volumecone}.

2.\ We will show that $\Omega^\circ\backslash \{p\}$ is connected. We argue by contradiction and outline the idea first. We can construct explicitly a Wasserstein geodesic between $\delta_{q}$ for a point $q \in \Omega^{\circ}\backslash (\Omega^*\cup \{p\})$ and another $\m$-absolutely continuous measure concentrated in $\Omega^*$. If $\Omega^\circ\backslash \{p\}$ is not connected but $\Omega^\circ$ is, this Wasserstein geodesic is concentrated on branching geodesics unless the metric measure space is isomorphic to an interval equipped with a measure that has a non-negative and semiconcave density w.r.t.\ $\mathcal L^1$. Since the former contradicts the essentially non-branching property of $\RCRD$ spaces, the latter must hold. But in this case the volume cone property of $\m$ on $\Omega^*$ contradicts that $\supp \m=X$ and that the $\mathcal L^1$ density of $\m$ must be positive in the interior of the interval because of semiconcavity.

Now, we give the precise construction.
Assume $\Omega^\circ\backslash \{p\}$ is not connected. Then we can pick ${q'}\in \Omega^{**}:=(\Omega^\circ\backslash \{p\}) \backslash (\Omega^*)$. One can see that $\Omega^{**}$ is open. Since $\Omega^{**}$ is open, it has positive measure. Thus we can pick $q'$ as above such that there is a unique arclength parameterized geodesic $r\in [0, d(p,q')] \mapsto \hat\gamma$ between $p =\gamma(0)$ and $q'$, and such that $\hat \gamma$ is inside $\Omega$. Choose $\delta>0$ small enough such that $B_{2\delta}(p)\subset \Omega$ and set $q:=\hat \gamma(\delta)$.

The set $\Omega^*\cap B_\delta(p)$ is open and hence has positive $\m$ measure. Let $\mu_0=\frac{1}{\m(\Omega^*\cap B_{\delta}(p))} \m|_{\Omega^*\cap B_\delta(p)}$ and $(\mu_t)_{t\in[0,1]}$ be the Wasserstein geodesic between $\mu_0$ and $\delta_q=\mu_1$.
Since we assume that $\Omega^{\circ} =\Omega^*\cup \Omega^{**}\cup \{p\} $ is connected, but $\Omega^*\cup \Omega^{**}=\Omega^{\circ}\backslash \{p\}$ is not connected, we have that the unique $L^2$-Wasserstein geodesic $(\mu_t)_{t\in [0,1]}$ (for uniqueness see \cite{Mon-Cav-17} for instance) must be given by $\mu_t= (e_t)_{\#}\Pi$
with $\Pi\in \mathcal P(\mathcal G(X))$ supported on geodesics that are reparametrized concatenations of the geodesic segments $r\in [0,\delta]\mapsto \gamma_\alpha(-r_{K, \kkapp(N-1), N}+\delta -r)$ and $\hat \gamma|_{[0,\delta]}$.
(Note that any geodesic that connects $q$ with a point in $\Omega^*$ must stay in $B_{2\delta}(p)\subset \Omega$ and since we assume $\Omega\backslash \{p\}$ is not connected, such geodesics must go through $p$. And for $\m$-a.e.\ $x$ in $\Omega^*\cap B_{\delta}(p)$ the unique geodesic that connects $x$ with $p$ is given by the restriction of some geodesic $\gamma_\alpha$ for $\alpha\in Q$.)

But $\RCRD$ spaces are essentially non-branching (see also Remark~\ref{remark_new} below). Hence it follows that $\Pi$ must be concentrated on non-branching geodesics. Hence, there exists a~single geodesic~$\tilde\gamma$ such that
$\Pi$ is concentrated on~$\tilde \gamma$. In particular $(\Omega^*\cap B_\delta(p), \m|_{\Omega^*\cap B_\delta(p)})$ is isomorphic to an interval. Then by \cite[Theorem~1.1]{kila} $(X,d,\m)$ is isomorphic to a $1$-dimensional manifold. Let us assume $p=0\in \R$ and $q=-d(0,q)=-\delta\in \R$ and $\Omega^*\subset (0,\infty)$. By the volume cone property that we proved in the previous step we have $\dm|_{\Omega^*}(r)$ is proportional to $\sin_{{K}/{(N-1)}}^{N-1}|r|\,{\rm d}r$. Since the density of $\m$ w.r.t.\ $\mathcal L^1$ must be semi-concave \cite[Theorem~1.1]{kila} and positive for interior points but must vanish at the origin, we have a contradiction.
Hence $\Omega^{\circ}\backslash \{p\}=\Omega^*$ is connected.

3.\ We can finish the proof of the main theorem by application of Theorem~\ref{thm:gp}. Recall that Theorem~\ref{thm:gp} provides a measure space isomorphism $U$ between $(\Omega^\circ,\m|_{\Omega^\circ})$ and a truncated cone such that~$U$ and its inverse are locally distance preserving. This yields that $U$ must be an isometry with respect to the induced intrinsic distances. Though this conclusion might be obvious to experts, we provide the proof in the following.

We observe first that
$S=\partial B_{r_{K,\kkapp(N-1),N}}(p)$ must contain more than two points:
our hypotheses rule out $S \subset \{pt\}$, while if $S$ consists of precisely two points,
Remark~\ref{rem:refined0} asserts $N=1$ and that
is also excluded by assumption.

Hence, only the last case in Theorem \ref{thm:gp} {remains relevant to us}.
By Remark \ref{rem:refined} there exist local isometries $ U$ and $ V$ between $\Omega^{\circ}=B_{r_{K,\kkapp(N-1),N}}(p)$ and the ball $B_{r_{K,\kkapp(N-1),N}}(o)$ in the corresponding cone that are also measure preserving bijections.

Now, it is standard knowledge that ${U}$ is an isometry with respect to the induced intrinsic distances.

Let us be more precise. Set $r_{K,\kkapp(N-1),N}=R$ and let $\tilde d_\epsilon$ be the induced intrinsic distance on $\bar B_{R-\epsilon}(p)$. We denote by $d^*$ the cone (or suspension distance) and by $\tilde d^*$ and $\tilde d^*_\epsilon$ the induced intrinsic distances of $B_R(o)$ and $\bar B_{R-\epsilon}(o)$, respectively. Then $ U$ is an isometry between $\bar B_{R-\epsilon}(p)$ and $\bar B_{R-\epsilon}(o)$ with respect to the induced intrinsic distances.

To prove this let $\gamma\colon [0,1]\rightarrow \bar B_{R-\epsilon}(p)$ be a geodesic with respect to $\tilde{d}_\epsilon$ between $x,y\in \bar B_{R-\epsilon}(p)$. We can
divide $\gamma$ into $k\in \mathbb N$ small pieces $\gamma|_{[t_{i-1},t_i]}$ with $i=1,\dots, k$ and $t_0=0$, $t_{k}=1$ such that each piece stays inside a~small ball that is mapped isometrically {with respect to} $d$ via
$ U$ to a~small ball in $B_R(o)$. We obtain
\begin{align*}
\sum_{i=1}^{k} d^*({U}(\gamma(t_{i-1})),{U}(\gamma(t_i)))=\sum_{i=1}^k d(\gamma(t_{i-1}),\gamma(t_i))\leq \sum_{i=1}^k\tilde d_\epsilon(\gamma(t_{i-1}), \gamma(t_i))= \tilde d_{\epsilon}(x,y).
\end{align*}
The first equality holds because $ U$ is an isometry {with respect to} $d$ and $d^*$ on the small balls that contain $\gamma|_{[t_{i-1},t_i]}$.
The last equality holds because $\gamma$ is geodesic {with respect to} $\tilde d_\epsilon$ and the inequality holds because the intrinsic distance is always equal or larger than $d$ itself.

On the left hand side we can take the supremum {with respect to} all such subdivisions $(t_i)_{i=0,\dots,k-1}$. This yields $\tilde{d}_{\epsilon}^*({U}(x),{U}(y))\leq L({U}\circ \gamma)\leq \tilde d_\epsilon(x,y)$ where $L(U\circ \gamma)$ is the length of the continuous curve $U\circ \gamma$. In particular ${U}\circ \gamma$ is a rectifiable curve{(that means has finite length)} in $\bar B_{R-\epsilon}(o)$.

We can argue in the same way for the inverse map ${V}$ and obtain that $U\colon \bar B_{R-\epsilon}(p)\rightarrow \bar B_{R-\epsilon}(o)$ is an isometry with respect to the induced intrinsic distances $\tilde d_\epsilon$ and $\tilde d^*_\epsilon$.

Finally, we let $\epsilon\rightarrow 0$ and observe that $\tilde d_\epsilon \rightarrow \tilde d$ on $\bar B_{R-\epsilon}(p)$ and the same for $\tilde d^*_\epsilon$ and $\tilde d^*$. This finishes the proof.

\begin{Remark}\label{remark_new}
A deep new result by Qin Deng \cite{qin} shows that $\RCRD$ spaces are in fact non-branching. In this case the middle step of the previous proof simplifies: If $\Omega^\circ\backslash \{p\}$ is not connected but $\Omega^\circ$ is connected, this yields almost immediately the existence of a branching geodesic unless the space is isomorphic to an interval equipped with a measure. The proof of Deng's result is quite long and involved. Therefore we provide a proof that only relies on the relatively weak property that the space is essentially non-branching.
\end{Remark}

\appendix
\section{Substituting measure contraction for lower Ricci bounds}\label{S:MCP}

In this appendix we will sketch why the results of Theorem~\ref{T:main} also hold when one replaces the condition $\CDD(K,N)$ with the weaker measure contraction property $\MCP(K,N)$ that was introduced in \cite{ohtmea, stugeo2}. We will not repeat the technical details but focus on necessary modifications for this setup.

For a proper metric measure space $(X,d,\m)$ that is essentially nonbranching there are several equivalent ways to define the $\MCP(K,N)$. The following one can be found in \cite[Section~9]{cavmil}.
\begin{Definition}[measure contraction property]
Let $(X,d,\m)$ be proper and essentially nonbranching. The measure contraction property $\MCP(K,N)$, $K\in \R$ and $N\in (1,\infty)$ holds if for every pair $\mu_0,\mu_1\in \mathcal P^2(X)$ such that $\mu_0$ is $\m$-absolutely continuous there exists a dynamical optimal plan $\Pi$ such that $(e_t)_{\#}\Pi=\mu_t\in \mathcal P^2(\m)$ and{\samepage
\begin{gather*}
\rho_t(\gamma_t)^{-\frac{1}{N}}\geq \tau_{K,N}^{(1-t)}({d(\gamma_0,\gamma_1)}) \rho_0(\gamma_{0})^{-\frac{1}{N}} \qquad \mbox{for }\Pi \mbox{-a.e.\ geodesic $\gamma$},
\end{gather*}
where $\mu_t=\rho_t\m$.}

For $\Omega\subset X$ with $\m(\Omega)>0$ the {\it restricted measure contraction property $\MCP_r$} is defined similarly as the condition $\CDD_r$ (compare with Definition \ref{def:cd}).
\end{Definition}
All the technical results in Section~\ref{subsec:1D} still hold when we replace the condition $\CDD(K,N)$ with $\MCP(K,N)$ (cf.~\cite{cav-mon-lapl-18}). Only in Theorem~\ref{th:1dlocalisation} the density $h_\alpha$ of the conditional measure $\m_\alpha$ need not satisfy~\eqref{kuconcave} and therefore need not be semiconcave, although it does remain locally Lipschitz on $(a(X_\alpha),b(X_\alpha))$ and extends continuously to the endpoints.
Instead one has only
\begin{align}\label{ineq:mcpmcp}
h_\alpha(\gamma_t)^{\frac{1}{N-1}}\geq \sigma_{K,N}^{(1-t)}({|\gamma_0-\gamma_1|}) h_{\alpha}(\gamma_0)^{\frac{1}{N-1}}
\end{align}
for every affine function $\gamma\colon [0,1]\rightarrow [a(X_\alpha),b(X_\alpha)]$ in general
where we consider $h_\alpha$ as a continuous function on $[a(X_\alpha),b(X_\alpha)]$.

Considering $\Omega\subset X$ with $\partial \Omega =S$ and $\m(S)=0$ then Definition~\ref{def:surfacemeasure} for $\m_S$ continues to make sense. We also can define the notion of finite inner curvature of~$\Omega$. However, since~$h_{\alpha}$ is not semiconcave in general, the right and the left derivative might not exist for every $t\in [a(X_\alpha), b(X_\alpha)]$.
Therefore, for a continuous function $f\colon [a,b]\rightarrow\mathbb R$ we set
\begin{gather*}
\frac{{\rm d}^-}{{\rm d}t}f(t)=\limsup_{h\uparrow 0} \frac{1}{h} \left[f(t+h)-f(t)\right] \qquad \mbox{for }t\in (a,b].
\end{gather*}
We can set up a definition of mean curvature for subsets in $\MCP$ spaces in the following way.
\begin{Definition}[inner mean curvature revisited in the $\MCP$ setting]\label{def:meanmcp}
Set $S=\partial \Omega$ and let $\{X_{\alpha}\}_{\alpha\in Q}$ be the disintegration induced by $u:=d_S$.
Recalling~\eqref{Binout}, we say that $S$ has {\em finite inner} (respectively {\em outer}) {\em curvature} if $\m\big(B^\dagger_{\rm in}\big)=0$ (respectively $\m(B^\dagger_{\rm out})=0$). If~$S$ has finite inner curvature we define the {inner mean curvature} of $S$ $\m_S$-almost everywhere as
\begin{gather} \label{inner upper}
p\in S\mapsto
H_S^-(p):= \begin{cases}
\dfrac{{\rm d}^-}{{\rm d}r}|_{r=0}\log h_\alpha\circ \gamma_\alpha
& \mbox{if}
\ p=\gamma_\alpha(0)\in S\cap A^{\dagger},\\
\infty& \mbox{if} \ p\in B^\dagger_{{\rm out}, e}\cap S,
\end{cases}
\end{gather}
where we set
$\frac{{\rm d}^-}{{\rm d}r} \log h_{\alpha}(\gamma_\alpha(0))= -\infty$ if $h_\alpha(\gamma_\alpha(0))=0$.
\end{Definition}

\begin{Theorem}[inscribed radius bounds under $\MCP$]\label{T:MCP}
Let $(X,d,\m)$ be an essentially nonbranching $\MCP(K',N)$ space with $K'\in \R$, $N\in {(}1,\infty)$ and $\supp \m=X$. Consider $K, H\in \R$ such that $\big(\frac{K}{N-1}, \frac{H}{N-1}\big)$ satisfies the ball condition. Let $\Omega\subset X$ be closed with {$\Omega\neq X$}, $\m(\Omega)>0$ and $\m(\partial \Omega)=0$ such that $\Omega$ satisfies the {\it restricted curvature-dimension condition} $\MCP_{r}(K,N)$ for $K\in \R$ and $\partial \Omega=S$ has finite inner curvature. Assume the inner mean curvature $H_S^-$ satisfies $H^-_S\geq H$ $\m_S$-a.e., where $\m_S$ denotes the surface measure. Then
\begin{gather*}
\inrad \Omega\leq r_{K,H,N},
\end{gather*}
where $\inrad\Omega=\sup\limits_{x\in \Omega} d_{\Omega^c}(x)$ is the inscribed radius of~$\Omega$.
\end{Theorem}
\begin{proof}Let $\eta\colon [0,1]\rightarrow [a(X_\alpha),0]$ be an affine function with $a(X_\alpha)<\eta_1<\eta_0 :=0$. We note that $|\eta_0-\eta_1|= d(\gamma_\alpha(\eta_0), \gamma_\alpha(\eta_1))$.
From \eqref{ineq:mcpmcp} we have
\begin{align*}
h_{\alpha}(\eta(t))^{\frac{1}{N-1}}\geq \sigma_{K/N-1}^{(1-t)}(|\eta_0-\eta_1|) h_{\alpha}(\eta_0)^{\frac{1}{N-1}}
\end{align*}for any $ \alpha\in Q^{\dagger}$. It
follows
\begin{align*}
\frac{{\rm d}^-}{{\rm d}r}\Big|_{r=\eta_0}h_\alpha & = \limsup_{\eta_t\rightarrow \eta_0} \frac{1}{\eta_t-\eta_0}\left[h_\alpha(\eta_t)-h_\alpha(\eta_0)\right]\\
& \leq \frac{1}{\eta_1-\eta_0}\frac{{\rm d}}{{\rm d}t} \Big|_{t=0}\sigma_{K/N-1}^{(1-t)} (|\eta_0-\eta_1|)^{N-1}h_{\alpha}(\eta_0).
\end{align*}
If the inner mean curvature is bounded below by $H$, then
\begin{gather*}
H\leq \frac{1}{\eta_1-\eta_0}\frac{{\rm d}}{{\rm d}t}\Big|_{t=0} \sigma_{K/N-1}^{(1-t)}(|\eta_0-\eta_1|)^{{N-1}}
 = \frac{-|\eta_0-\eta_1|}{\eta_1-\eta_0} (N-1) \frac{\cos_{K/(N-1)}(|\eta_0-\eta_1|)}{\sin_{K/(N-1)}(|\eta_0-\eta_1|)}.
\end{gather*}
Hence $\frac{H}{N-1}\leq \frac{\cos_{K/(N-1)}(|\eta_0-\eta_1|)}{\sin_{K/(N-1)}(|\eta_0-\eta_1|)}$ for $\mathfrak q$-a.e.\ $\alpha\in Q^{\dagger}$.
The sharp Bonnet--Myers diameter bound \cite{lottvillani,stugeo2} for $\CDD(K,N)$ spaces yields $|\eta_0 -\eta_1| < \pi_{\frac K{N-1}}$
 from~\eqref{pik}. Thus the denominator above is non-negative and
\begin{align*}
0\leq \cos_{K/(N-1)}(|\eta_0-\eta_1|)-\frac{H}{N-1}\sin_{K/(N-1)}(|\eta_0-\eta_1|).
\end{align*}
Since this expression holds for all $\eta_1 \in (a(X_\alpha),0)$ we conclude $d(\gamma_\alpha(\eta_0),\gamma_\alpha(\eta_1))=|\eta_0-\eta_1|\leq r_{K,H,N}$. Otherwise $|\eta_0-\eta_1|>r_{K,H,N}$ implies that the right hand side in the last inequality is negative by the definition of $r_{K,H,N}$.

Noting $\eta_0=0$, taking $\eta_1\rightarrow a(X_\alpha)$, one can finish the argument exactly as for $\CDD(K,N)$ spaces.
\end{proof}

\section{A {different} form of mean curvature bound also suffices}\label{S:backwards}

Inspired by \cite{cm_new}, in this appendix we introduce another new notion of mean curvature bounded from below which yields Theorems~\ref{T:main} and~\ref{T:main3}
without requiring finite inner curvature of $\Omega$ but assuming that the measure $\mathfrak p_0$ in Remark~\ref{rem:surmea} is a Radon measure on $Q$.
 Let $L_{-{\rm loc}}^{1}(Q, {\rm d}\mathfrak p_0)$ denote the class of $\mathfrak p_0$-measurable functions $k\colon Q\longrightarrow [-\infty,\infty]$ whose negative part $\min\{0,k\}$
 belongs to $L^1_{\rm loc}(Q, {\rm d}\mathfrak p_0)$, i.e., is locally
$\mathfrak p_0$-summable.

\begin{Definition}[backward mean curvature bounded below] Let $(X,d,\m)$ be an essentially nonbranching metric measure space that satisfies $\MCP$ or $\CDD$. Recall the family of measures $\{\mathfrak p_t \}_{t \in (-\infty,0]}$ on $Q$ given by
${\rm d}\mathfrak p_t(\alpha)= h_\alpha\circ g(\alpha,t) {\rm d}\mathfrak q|_{\mathcal V_t}(\alpha)$
that we introduced in Remark~\ref{rem:surmea}, and its image $\m_{S_t}=g(\cdot,t)_\#\mathfrak p_t$ on~$X$.
Recall that $Q$ is constructed as a Borel subset of $X$ and $(\alpha,t)\mapsto g(\alpha,t)$ is the ray map
constructed in that remark.

Then $S=\partial \Omega$ has \emph{backward mean curvature bounded from below by $k \in L^{1}_{-{\rm loc}}(Q,{\rm d}\mathfrak p_0)$} if the measure $\mathfrak p_0$ is a Radon measure, $h_\alpha\circ g(\alpha,0)>0$ for $\mathfrak q$-a.e.\ $\alpha \in Q$, and
\begin{align}\label{ineq:bwmc}
\frac{{\rm d}^-}{{\rm d}t} \Big|_{t=0} \int_\Y {\rm d}\mathfrak p_t
:= \limsup_{t\uparrow 0} \frac{1}{t}\left( \int_\Y {\rm d}\mathfrak p_t- \int_\Y {\rm d}\mathfrak p_0\right)\geq \int_\Y k(q) \,{\rm d}\mathfrak p_0(q)
\end{align}
for any bounded measurable subset $Y\subset Q$.
Moreover, $S$ has \emph{backward-lower} mean curvature bounded from below by $k$ if the same inequality holds when $\limsup$ is replaced by $\liminf$.
Heuristically, these limits quantify the relative rate of change of surface area of the level sets of $d_S$, as when $Y=Q \subset X$ is bounded.
We may denote the greatest lower bounds $k^\pm \in L^{1}_{-{\rm loc}}(Q,{\rm d}\mathfrak p_0)$ for the backwards and backwards-lower mean curvature for $k^+$ and $k^-$
respectively.
\end{Definition}

\begin{Remark}
Since it is not assumed that $\mathfrak p_t$ for $t<0$ is a Radon measure, $\int_\Y {\rm d}\mathfrak p_t$ can be infinite.
\end{Remark}

Similarly, when we want to distinguish between upper and lower limits, we refer to~\eqref{inner upper} as the {\em inner mean curvature},
and to the analogous quantity with $\liminf$ in place of $\limsup$, as the {\em inner-lower mean curvature.}
In a $\CDD(K,N)$ space, these two notions coincide.

\begin{Lemma}[backward versus inner mean curvature in $\MCP$]\label{lem:bmc} Let $(X,d,\m)$ be an essentially nonbranching $\MCP(K,N)$ space. Assume that $S=\partial \Omega$ for some Borel set $\Omega$ that has finite inner curvature in the sense of Definition~{\rm \ref{def:meanmcp}} and that $h_\alpha\circ g(\alpha,0) {\rm d}\mathfrak q(\alpha)= {\rm d}\mathfrak p_0(\alpha)$ is a Radon measure with $h_\alpha\circ g(\alpha,0)>0$ for $\mathfrak q$-a.e.\ $\alpha\in Q$.
Then $\m_{S}$-almost everywhere,
the inner mean curvature $H^-_S(x)$ of $S$ is bounded from below by $k(x)$ if $S$ has backward mean curvature bounded from below by $k \circ g(\cdot,0) \in L^{1}_{-{\rm loc}}(Q,{\rm d} \mathfrak p_0)$.

Conversely, the backward-lower mean curvature is bounded from below by $k \circ g(\cdot,0) $ if
$k\circ g(\cdot,0) \in L^{1}_{-{\rm loc}}(Q,{\rm d} \mathfrak p_0)$ and there exists $\delta>0$ such that
$(i)$~the inner-lower mean curvature of~$S$ is bounded from below by $k(x)$ for $\m_S$-almost every $x\in S$, and
$(ii)$~$b(X_\alpha)>\delta$ holds for $q$-$a.e.\ \alpha \in Q$. Note $(ii)$ is satisfied if $\Omega$ satisfies a uniform exterior ball condition.
\end{Lemma}

\begin{proof} {\bf 1.}
Assume backward mean curvature bounded below by $k\circ g(\cdot,0)\in L^{1}_{-{\rm loc}}(Q,{\rm d}\mathfrak p_0)$. Then, by monotonicity of the right hand side in~\eqref{ineq:bwmc} the backward mean curvature is also bounded below by $k^M:=\min\{k, M\} \in L^1_{\rm loc}(Q, {\rm d}\mathfrak p_0)$ for $M>0$ arbitrary.

We can compute for $t<0$ and a bounded measurable set $Y\subset Q$:
\begin{align*}
 \int_\Y {\rm d}\mathfrak p_t- \int_Y {\rm d}\mathfrak p_0
 = \int_\Y\left(1_{\mathcal V_t}(\alpha) h_\alpha\circ g(\alpha,t) - 1_{\mathcal V_0}(\alpha) h_\alpha\circ g(\alpha,0)\right) {\rm d}\mathfrak q(\alpha).
 \end{align*}
There exists a measurable subset $Q^*\subset Q^\dagger$ with $\mathfrak q[Q^\dagger\setminus Q^*]=0$ such that the map $\mathcal M\colon \alpha\in Q^* \mapsto -a(X_\alpha)\in [0,\infty)$ is measurable (for instance compare with step 1 in the proof of Theorem~7.10 in~\cite{cavmil} or Remark~3.4 in~\cite{kks}). Then, we consider the family of measurable sets $Q_{m}=\mathcal M^{-1}\big(\big[\frac{1}{m},m\big]\big)$ for $m\in \mathbb N$ that satisfy $\bigcup_{m\in \mathbb N}Q_{m}=Q^*$. As in~\cite{cavmil}
\begin{gather*}
\frac{1}{h_\alpha\circ g(\alpha,0)}\frac{1}{r}(h_\alpha\circ g(\alpha,r) - h_\alpha\circ g(\alpha,0))
\leq (N-1)\frac{\cos_{-|K|/(N-1)}(-a(X_\alpha))}{\sin_{-|K|/(N-1)}(-a(X_\alpha))}\leq C(K,N,m)\!
\end{gather*} $\forall\, r\in (a(X_\alpha),0), \ \forall\, \alpha\in Q_{m}$.
Thus Fatou's lemma yields
\begin{gather*}
\int_{\Y\cap Q_{m} \cap \mathcal V_0} \big(k^Mh_\alpha\big)\circ g(\alpha,0) \,{\rm d}\mathfrak q(\alpha)=\int_{\Y\cap Q_{m}} k^M \circ g(\alpha ,0) \,{\rm d}\mathfrak p_0(\alpha)\\
\qquad{}\leq \int_{\Y\cap Q_{m}} \limsup_{t\uparrow 0}\frac{1}{t} \left(1_{\mathcal V_t}(\alpha) h_\alpha\circ g(\alpha,t) - 1_{\mathcal V_0}(\alpha) h_\alpha\circ g(\alpha,0)\right) {\rm d}\mathfrak q(\alpha)
 \\
\qquad{}\leq \int_{\Y\cap Q_{m}} \limsup_{t\uparrow 0}\frac{1}{t} \left(1_{\mathcal V_t\cap \mathcal V_0}(\alpha) h_\alpha \circ g(\alpha,t) - 1_{\mathcal V_0}(\alpha) h_\alpha\circ g(\alpha,0)\right) {\rm d}\mathfrak q(\alpha)
 \\
\qquad{} =\int_{\Y\cap Q_{m}\cap \mathcal V_0} \frac{{\rm d}^-}{{\rm d}t}\Big|_{t=0} h_\alpha \circ g(\alpha,t) \,{\rm d}\mathfrak q(\alpha)
 \end{gather*}
 for any bounded measurable $Y \subset Q$.
 It follows that
 \begin{gather}\label{inequ:nini}
 \big(k^M h_\alpha\big)\circ g(\alpha,0)\leq \frac{{\rm d}^-}{{\rm d}t}\Big|_{t=0}h_\alpha\circ g(\alpha,t) \qquad \mbox{for }\mathfrak q\mbox{-a.e. }\alpha\in \mathcal V_0.
 \end{gather}
We assumed $h_\alpha\circ g(\alpha,0)>0$ for $\mathfrak q$-almost every $\alpha\in \mathcal V_0$.
Therefore, it follows that
 \begin{align*}
\int_{\Y \cap \mathcal V_0} \big(k^{M}h_\alpha\big)\circ g(\alpha,0)\,{\rm d}\mathfrak q(\alpha)
&\leq \int_{\Y \cap \mathcal V_0} \frac{{\rm d}^-}{{\rm d}t}\Big|_{t=0} h_\alpha\circ g(\alpha,t) \,{\rm d}\mathfrak q(\alpha)\\
&=\int_{\Y \cap \mathcal V_0} \frac{{\rm d}^-}{{\rm d}t}\Big|_{t=0} \log (h_\alpha\circ g(\alpha,t)) h_\alpha\circ g(\alpha,0) \,{\rm d}\mathfrak q(\alpha).
 \end{align*}
Now, we recall that $\mathcal V_0\subset \mathfrak Q\big(A^\dagger\big)\cup B_{\rm in}^{\dagger}$ with $\mathfrak q\big( \mathfrak Q\big(A^{\dagger}\cup B_{\rm in}^{\dagger}\big)\backslash \mathcal V_0\big)=0$ and $\m\big(B^{\dagger}_{\rm in}\big)=\mathfrak q\big(\mathfrak Q\big(B^{\dagger}_{\rm in}\big)\big)=0$ (because we assume finite inner curvature). Moreover $g(\alpha,t)=\gamma_\alpha(t)$, $h_\alpha\circ g(\alpha,t)=h_\alpha\circ \gamma_\alpha(t)=h_\alpha(t)$ and $\alpha=\gamma_\alpha(0)$ if $\alpha\in \mathfrak Q(A)$. Hence $\m_S=\m_{S_0}= g(\cdot,0)_{\#} \mathfrak p_0$ and
\begin{align*}
\int_\Y k^{M}\dm_S &
=\int_{\Y \cap \mathfrak Q(A^\dagger)}
k^{M}\circ \gamma_\alpha(0) h_\alpha(0)\,{\rm d}\mathfrak q(\alpha) \\
&\leq \int_{\Y \cap \mathfrak Q(A^\dagger)}
H^-_S\circ \gamma_\alpha(0) h_\alpha(0) \,{\rm d}\mathfrak q(\alpha)
=\int_\Y H^-_S \dm_S,
 \end{align*}
and consequently $k^{M}\leq H_S^-$ $\m_S$-almost everywhere. Letting $M\rightarrow \infty$ yields $k\leq H_S^-$ $\m_S$-almost everywhere.

{\bf 2.} Now assume for some $k\circ g(\cdot,0)\in L^{1}_{-{\rm loc}}(Q,{\rm d} \mathfrak p_0)$
that $\m_S$-a.e.\ the inner-{\em lower} mean curvature is bounded from below by $k$. We also assume that $b(X_\alpha)>\delta>0$ for $\mathfrak q$-almost every $\alpha\in Q$ and some $\delta>0$. This implies $B_{\rm in} =\varnothing$ and hence $\m_{S_0}=\m_S$.
Recall from the proof of Lemma~\ref{lem:ballcondition} that this is true if $\Omega$ satisfies a uniform exterior ball condition.

Again as in \cite{cavmil}
\begin{align*}
-C(K,N,\delta) \leq-(N-1)\frac{\cos_{-|K|/(N-1)}(b(X_\alpha)-a(X_\alpha))}{\sin_{-|K|/(N-1)}(b(X_\alpha)-a(X_\alpha))}\leq
\frac{(h_\alpha\circ g(\alpha,r) - h_\alpha\circ g(\alpha,0))}{r \ \cdot \ h_\alpha\circ g(\alpha, 0)}
\end{align*}$\forall\, r\in (a(X_\alpha),b(X_\alpha))$.

Using Fatou's lemma again, it follows that
\begin{align*}
\int_{Y} k \circ \gamma_\alpha(0)\, {\rm d}\mathfrak p_0(\alpha)
&=\int_{Y \cap \mathcal V_0} (kh_\alpha)\circ g(\alpha,0) \,{\rm d}\mathfrak q(\alpha)
 =\int_{Y \cap \mathfrak Q(A^\dagger)} k(\gamma_\alpha(0)) h_\alpha(0) \,{\rm d}\mathfrak q(\alpha)\\
&\leq \int_{Y \cap \mathcal V_0} \liminf_{t \uparrow 0} \frac{\log h_\alpha(t) - \log h_\alpha(0)}{t} h_\alpha(0) \,{\rm d}\mathfrak q(\alpha)\\
&\leq \liminf_{t\uparrow 0} \int_{Y} \frac{1}{t} \left(1_{\mathcal V_t}(\alpha) h_\alpha(t) - 1_{\mathcal V_0}(\alpha) h_\alpha(0)\right) {\rm d}\mathfrak q(\alpha)\\
&= \liminf_{t\uparrow 0} \frac{1}{t}\left(
 \int_{Y} {\rm d}\mathfrak p_t- \int_{Y} {\rm d}\mathfrak p_0\right),
 \end{align*}
 for any bounded measurable set $Y\subset Q$. Thus the backward-lower mean curvature is bounded below by $k\circ g(\cdot,0)$ as desired.
\end{proof}

\begin{Corollary}[backward versus inner mean curvature in $\CDD$]\label{cor:bmc} Let $(X,d,\m)$ be an essentially nonbranching $\CDD(K,N)$ space. Assume that $S=\partial \Omega$ for some Borel set $\Omega$ satisfies a uniform external ball condition and that $h_\alpha\circ g(\alpha,0) {\rm d}\mathfrak q(\alpha)= {\rm d}\mathfrak p_0(\alpha)$ is a Radon measure with \mbox{$h_\alpha\circ g(\alpha,0)>0$} for $\mathfrak q$-a.e.\ $\alpha\in Q$.
Then $S$ has a backwards mean curvature bound $k \in L^{1}_{-{\rm loc}}(Q,{\rm d}\mathfrak p_0)$ from below if and only if $S$ has inner mean curvature $H^-_S \in L^{1}_{-{\rm loc}}(S,\dm_S)$.
When either holds, then $\mathfrak p_0$-almost everywhere on $Q$,
the backward and backward-lower mean curvatures $k^\pm$ of $S$ both coincide with $H^-_S \circ g(\cdot,0)$.
\end{Corollary}

\begin{proof}Recall that the uniform external ball condition assumed implies $S$ has finite inner curvature by Lemma~\ref{lem:ballcondition}, and that $\m_S=\m_{S_0}$.
Let $H^-_S$ denote the inner mean curvature of $S$, which agrees with its inner-lower mean curvature $\m_S$-a.e.\ due to the semiconcavity of $h_\alpha^{1/{(N-1)}}$ in
$\CDD(K,N)$ spaces.

If $S$ has a backwards mean curvature bound $k \circ g(\cdot,0) \in L^{1}_{-{\rm loc}}(Q,{\rm d}\mathfrak p_0)$ from below,
it admits a greatest such bound $k^+ \circ g(\cdot,0)$.
Lemma~\ref{lem:bmc} asserts $H^-_S\ \circ g(\cdot,0) \ge k^+\circ g(\cdot,0)$ holds $\mathfrak p_0$-a.e., which implies $H^-_S \in L^1_{-{\rm loc}}(S,\dm_S)$.

Conversely, if $S$ has inner mean curvature $H^-_S \in L^1_{-{\rm loc}}(S,\dm_S)$, then Lemma~\ref{lem:bmc} asserts $S$ has backwards-lower mean curvature
$k^-\circ g(\cdot,0) \ge H^-_S \circ g(\cdot,0) \in L^1_{-{\rm loc}}(Q,{\rm d}\mathfrak p_0)$. We conclude $k^-\circ g(\cdot,0)$ is also a backwards (i.e., backwards-upper)
mean curvature lower bound for $S$.

Since $k^+ \circ g(\cdot,0) \ge k^-\circ g(\cdot,0)$ by definition, in either (and hence both) cases above we conclude equalities hold $\mathfrak p_0$-a.e.\ in all three of the inequalities
preceding, to conclude the proof.
\end{proof}

We {state a theorem under $\MCP$. The corresponding statement for $\CDD$ then follows} since $\CDD$ implies $\MCP$ for essentially nonbranching proper metric measure spaces.

\begin{Theorem}[inradius bounds under backward mean curvature bounded below]\label{T:main4}
Let $(X,d,\m)$ be an essentially nonbranching $\MCP(K,N)$ space with $K\in \R$, $N\in {(}1,\infty)$ and $\supp \m=X$. Let $\Omega\subset X$ be closed with {$\Omega\neq X$}, $\m(\Omega)>0$ and $\m(\partial \Omega)=0$. Assume $\partial \Omega=S$ has backward mean curvature bounded from below by $H\in \mathbb R$.
Then
\begin{align*}
\inrad \Omega\leq r_{K,H,N}.
\end{align*}
\end{Theorem}
\begin{proof}
As in the {previous} appendix we have
\begin{align*}
h_{\alpha}(t)^{\frac{1}{N-1}}\geq \sigma_{K/N-1}^{(1-\frac{t}{a(X_\alpha)})}( -a(X_\alpha)) h_{\alpha}(0)^{\frac{1}{N-1}}
\end{align*}$\mbox{for any }t\in (a(X_\alpha), 0) \mbox{and any } \alpha\in Q^{\dagger}$.
Therefore, it follows that
\begin{align*}
\frac{{\rm d}^-}{{\rm d}t}\Big|_{t=0}h_{\alpha}\circ g(\alpha,t)
&=\limsup_{t\uparrow 0}\frac{1}{t} \left(h_\alpha(g(\alpha,t))- h_\alpha(g(\alpha,0))\right)
\\&\leq \frac{{\rm d}^-}{{\rm d}t} \Big|_{t=0}\sigma_{K/N-1}^{\left(\frac{a(X_\alpha)-t}{a(X_\alpha)}\right)} (-a(X_\alpha))^{N-1}h_{\alpha}(g(\alpha,0)).
\end{align*}
Since the backward mean curvature is bounded below by $H$, for $Y \subset Q$ bounded and measurable it follows that
\begin{gather*}%\label{zzz}
H\int_{\Y \cap \mathcal V_0} h_\alpha\circ g(\alpha,0) \,{\rm d}\mathfrak q(\alpha)
\leq \int_{\Y \cap \mathcal V_0} \frac{{\rm d}^-}{{\rm d}t} \Big|_{t=0} h_\alpha\circ g(\alpha,t) \,{\rm d}\mathfrak q(\alpha).
\end{gather*}
We obtain the inequality \eqref{inequ:nini} exactly as in the beginning of step 1 of the proof of Lemma~\ref{lem:bmc}. By the definition of backward-lower mean curvature bounds we have $h_\alpha(0)=h_\alpha\circ g(\alpha,0)>0$ for $\mathfrak q$-almost every $\alpha$.
Hence
\[ \frac{H}{N-1}\leq \frac{1}{N-1} \frac{{\rm d}^-}{{\rm d}t}\Big|_{t=0}\log h_\alpha\circ g(\alpha,t)\leq \frac{\cos_{K/(N-1)}(-a(X_\alpha))}{\sin_{K/(N-1)}({- a(X_{\alpha}) })}
\] for $\mathfrak q$-a.e.\ $\alpha\in \mathcal V_0=\mathfrak Q(A\cup B_{\rm in})$.

At this point it is clear that we can finish the proof as in Theorem \ref{T:MCP}.
\end{proof}

\begin{Theorem}[rigidity under backward mean curvature bounded from below]\label{T:main5}
Let $(X,d,\m)$ be $\RCRD(K,N)$ for $K\in \R$ and $N\in (1,\infty)$ and let $\Omega\subset X$ be compact with {$\Omega\neq X$}, $\m(\Omega)>0$, connected and non-empty interior $\Omega^{\circ}$ and $\m(\partial \Omega)=0$. We assume that $K\in \{N-1, 0, -(N-1)\}$, $\partial \Omega=S\neq \{pt\}$ and $S$ has backward mean curvature bounded below by $\kkapp(N-1)\in \mathbb R$.
 Then, there exists $x\in X$ such that
 \[ d_S(x)=\inrad \Omega =r_{K,\kkapp(N-1),N}\] if and only if $r_{K,\chi(N-1), N}<\infty$
and there exists an $\RCRD(N-2,N-1)$ space $Y$ such that $\big(\Omega^\circ, \tilde d_{\Omega^\circ}\big)$ is isometric to
 $\big(B_{r_{K,\kkapp(N-1),N}}(0), \tilde d\big)$ in $\tilde I_{\frac{K}{N-1}}\times_{\sin_{K/(N-1)}}^{N-1} Y$, where $\tilde d_{\Omega}$ and $\tilde d$ are the induced intrinsic distances of $\Omega^\circ$ and ${B}_{r_{K,\kkapp(N-1),N}}(0)$, respectively.
\end{Theorem}
\begin{proof}
In the end of the proof of Theorem \ref{T:main4} we obtained $H\leq \frac{{\rm d}^-}{{\rm d}t} \big|_{t=0}\log h_\alpha\circ g(\alpha,t)$ for $\mathfrak q$-a.e.\ $\alpha\in \mathcal V_0$ with $H=\kkapp (N-1)$, so in particular for $\alpha\in \mathfrak Q(B_{\rm in})$. Then using the Riccati comparison and the maximum principle we can
follow verbatim the same proof as in Section \ref{sec:rig}.
\end{proof}

\subsection*{Acknowledgements}
 The authors are grateful to Yohei Sakurai for directing us to the work of Kasue,
and to the anonymous referees for very constructive comments.
AB is supported by the Dutch Research Council (NWO)~-- Project number VI.Veni.192.208.
CK is funded by the Deutsche Forschungsgemeinschaft (DFG) -- Projektnummer 396662902, ``Synthetische Kr\"ummungsschranken durch Methoden des optimal Transports''.
RM's research is supported in part by NSERC Discovery Grants RGPIN--2015--04383 and 2020--04162.
EW's research is supported in part by NSERC Discovery Grant RGPIN-2017-04896.

\pdfbookmark[1]{References}{ref}
\LastPageEnding


\begin{thebibliography}{99}
\footnotesize\itemsep=0pt

\bibitem{agmr}
Ambrosio L., Gigli N., Mondino A., Rajala T., Riemannian {R}icci curvature
 lower bounds in metric measure spaces with {$\sigma$}-finite measure,
 \href{https://doi.org/10.1090/S0002-9947-2015-06111-X}{\textit{Trans. Amer. Math. Soc.}} \textbf{367} (2015), 4661--4701,
 \href{https://arxiv.org/abs/1207.4924}{arXiv:1207.4924}.

\bibitem{agslipschitz}
Ambrosio L., Gigli N., Savar\'e G., Density of {L}ipschitz functions and
 equivalence of weak gradients in metric measure spaces, \href{https://doi.org/10.4171/RMI/746}{\textit{Rev. Mat.
 Iberoam.}} \textbf{29} (2013), 969--996, \href{https://arxiv.org/abs/1111.3730}{arXiv:1111.3730}.

\bibitem{agsheat}
Ambrosio L., Gigli N., Savar\'e G., Calculus and heat flow in metric measure
 spaces and applications to spaces with {R}icci bounds from below,
 \href{https://doi.org/10.1007/s00222-013-0456-1}{\textit{Invent. Math.}} \textbf{195} (2014), 289--391, \href{https://arxiv.org/abs/1106.2090}{arXiv:1106.2090}.

\bibitem{agsriemannian}
Ambrosio L., Gigli N., Savar\'e G., Metric measure spaces with {R}iemannian
 {R}icci curvature bounded from below, \href{https://doi.org/10.1215/00127094-2681605}{\textit{Duke Math.~J.}} \textbf{163}
 (2014), 1405--1490, \href{https://arxiv.org/abs/1109.0222}{arXiv:1109.0222}.

\bibitem{amsnonlinear}
Ambrosio L., Mondino A., Savar\'e G., Nonlinear diffusion equations and
 curvature conditions in metric measure spaces, \href{https://doi.org/10.1090/memo/1270}{\textit{Mem. Amer. Math. Soc.}}
 \textbf{262} (2019), v+121~pages, \href{https://arxiv.org/abs/1509.07273}{arXiv:1509.07273}.

\bibitem{bjoern}
Bj\"orn A., Bj\"orn J., Nonlinear potential theory on metric spaces,
 \textit{EMS Tracts in Mathematics}, Vol.~17, \href{https://doi.org/10.4171/099}{European Mathematical Society
 (EMS)}, Z\"urich, 2011.

\bibitem{bbi}
Burago D., Burago Y., Ivanov S., A course in metric geometry, \textit{Graduate
 Studies in Mathematics}, Vol.~33, \href{https://doi.org/10.1090/gsm/033}{Amer. Math. Soc.}, Providence, RI, 2001.

\bibitem{cavom}
Cavalletti F., Monge problem in metric measure spaces with {R}iemannian
 curvature-dimension condition, \href{https://doi.org/10.1016/j.na.2013.12.008}{\textit{Nonlinear Anal.}} \textbf{99} (2014),
 136--151, \href{https://arxiv.org/abs/1310.4036}{arXiv:1310.4036}.

\bibitem{cavmil}
Cavalletti F., Milman E., The globalization theorem for the curvature dimension
 condition, \href{https://arxiv.org/abs/1612.07623}{arXiv:1612.07623}.

\bibitem{Mon-Cav-17}
Cavalletti F., Mondino A., Optimal maps in essentially non-branching spaces,
 \href{https://doi.org/10.1142/S0219199717500079}{\textit{Commun. Contemp. Math.}} \textbf{19} (2017), 1750007, 27~pages,
 \href{https://arxiv.org/abs/1609.00782}{arXiv:1609.00782}.

\bibitem{cavmon}
Cavalletti F., Mondino A., Sharp and rigid isoperimetric inequalities in
 metric-measure spaces with lower {R}icci curvature bounds, \href{https://doi.org/10.1007/s00222-016-0700-6}{\textit{Invent.
 Math.}} \textbf{208} (2017), 803--849, \href{https://arxiv.org/abs/1502.06465}{arXiv:1502.06465}.

\bibitem{cav-mon-lapl-18}
Cavalletti F., Mondino A., New formulas for the {L}aplacian of distance
 functions and applications, \href{https://doi.org/10.2140/apde.2020.13.2091}{\textit{Anal. PDE}} \textbf{13} (2020),
 2091--2147, \href{https://arxiv.org/abs/1803.09687}{arXiv:1803.09687}.

\bibitem{cm_new}
Cavalletti F., Mondino A., Optimal transport in Lorentzian synthetic spaces,
 synthetic timelike {R}icci curvature lower bounds and applications,
 \href{https://arxiv.org/abs/2004.08934}{arXiv:2004.08934}.

\bibitem{cheegerlipschitz}
Cheeger J., Differentiability of {L}ipschitz functions on metric measure
 spaces, \href{https://doi.org/10.1007/s000390050094}{\textit{Geom. Funct. Anal.}} \textbf{9} (1999), 428--517.

\bibitem{Cheng75}
Cheng S.Y., Eigenvalue comparison theorems and its geometric applications,
 \href{https://doi.org/10.1007/BF01214381}{\textit{Math.~Z.}} \textbf{143} (1975), 289--297.

\bibitem{CushingKamtueKoolenLiuMunchPeyerinhoff20}
Cushing D., Kamtue S., Koolen J., Liu S., M\"unch F., Peyerimhoff N., Rigidity
 of the {B}onnet--{M}yers inequality for graphs with respect to {O}llivier
 {R}icci curvature, \href{https://doi.org/10.1016/j.aim.2020.107188}{\textit{Adv. Math.}} \textbf{369} (2020), 107188, 53~pages,
 \href{https://arxiv.org/abs/1807.02384}{arXiv:1807.02384}.

\bibitem{DGi}
De~Philippis G., Gigli N., From volume cone to metric cone in the nonsmooth
 setting, \href{https://doi.org/10.1007/s00039-016-0391-6}{\textit{Geom. Funct. Anal.}} \textbf{26} (2016), 1526--1587,
 \href{https://arxiv.org/abs/1512.03113}{arXiv:1512.03113}.

\bibitem{qin}
Deng Q., H\"older continuity of tangent cones in ${\rm RCD}(K,N)$ spaces and
 applications to non-branching, \href{https://arxiv.org/abs/2009.07956}{arXiv:2009.07956}.

\bibitem{erbarkuwadasturm}
Erbar M., Kuwada K., Sturm K.-T., On the equivalence of the entropic
 curvature-dimension condition and {B}ochner's inequality on metric measure
 spaces, \href{https://doi.org/10.1007/s00222-014-0563-7}{\textit{Invent. Math.}} \textbf{201} (2015), 993--1071,
 \href{https://arxiv.org/abs/1303.4382}{arXiv:1303.4382}.

\bibitem{EvansGangbo99}
Evans L.C., Gangbo W., Differential equations methods for the
 {M}onge--{K}antorovich mass transfer problem, \href{https://doi.org/10.1090/memo/0653}{\textit{Mem. Amer. Math. Soc.}}
 \textbf{137} (1999), viii+66~pages.

\bibitem{FeldmanMcCann02}
Feldman M., McCann R.J., Monge's transport problem on a {R}iemannian manifold,
 \href{https://doi.org/10.1090/S0002-9947-01-02930-0}{\textit{Trans. Amer. Math. Soc.}} \textbf{354} (2002), 1667--1697.

\bibitem{fremlin}
Fremlin D.H., Measure theory, {V}ol.~4, Topological measure spaces, Part~I,~II,
 Torres Fremlin, Colchester, 2006.

\bibitem{giglistructure}
Gigli N., On the differential structure of metric measure spaces and
 applications, \href{https://doi.org/10.1090/memo/1113}{\textit{Mem. Amer. Math. Soc.}} \textbf{236} (2015),
 vi+91~pages, \href{https://arxiv.org/abs/1205.6622}{arXiv:1205.6622}.

\bibitem{giglimondino}
Gigli N., Mondino A., A {PDE} approach to nonlinear potential theory in metric
 measure spaces, \href{https://doi.org/10.1016/j.matpur.2013.01.011}{\textit{J.~Math. Pures Appl.}} \textbf{100} (2013), 505--534,
 \href{https://arxiv.org/abs/1209.3796}{arXiv:1209.3796}.

\bibitem{gigli_rigoni}
Gigli N., Rigoni C., A note about the strong maximum principle on {RCD} spaces,
 \href{https://doi.org/10.4153/cmb-2018-022-9}{\textit{Canad. Math. Bull.}} \textbf{62} (2019), 259--266,
 \href{https://arxiv.org/abs/1706.01998}{arXiv:1706.01998}.

\bibitem{Graf19+}
Graf M., Singularity theorems for {$C^1$}-{L}orentzian metrics, \href{https://doi.org/10.1007/s00220-020-03808-y}{\textit{Comm.
 Math. Phys.}} \textbf{378} (2020), 1417--1450, \href{https://arxiv.org/abs/1910.13915}{arXiv:1910.13915}.

\bibitem{Hawking66}
Hawking S.W., The occurrence of singularities in cosmology.~{I}, \href{https://doi.org/10.1098/rspa.1966.0221}{\textit{Proc.
 Roy. Soc. London Ser.~A}} \textbf{294} (1966), 511--521.

\bibitem{kks}
Kapovitch V., Ketterer C., Sturm K.-T., On gluing {A}lexandrov spaces with lower
 {R}icci curvature bounds, \href{https://arxiv.org/abs/2003.06242}{arXiv:2003.06242}.

\bibitem{kasue-laplace}
Kasue A., A {L}aplacian comparison theorem and function theoretic properties of
 a complete {R}iemannian manifold, \href{https://doi.org/10.4099/math1924.8.309}{\textit{Japan.~J. Math. (N.S.)}} \textbf{8}
 (1982), 309--341.

\bibitem{Kasue83}
Kasue A., Ricci curvature, geodesics and some geometric properties of
 {R}iemannian manifolds with boundary, \href{https://doi.org/10.2969/jmsj/03510117}{\textit{J.~Math. Soc. Japan}}
 \textbf{35} (1983), 117--131.

\bibitem{ketterer}
Ketterer C., Ricci curvature bounds for warped products, \href{https://doi.org/10.1016/j.jfa.2013.05.008}{\textit{J.~Funct.
 Anal.}} \textbf{265} (2013), 266--299, \href{https://arxiv.org/abs/1209.1325}{arXiv:1209.1325}.

\bibitem{ketterer2}
Ketterer C., Cones over metric measure spaces and the maximal diameter theorem,
 \href{https://doi.org/10.1016/j.matpur.2014.10.011}{\textit{J.~Math. Pures Appl.}} \textbf{103} (2015), 1228--1275,
 \href{https://arxiv.org/abs/1311.1307}{arXiv:1311.1307}.

\bibitem{Ketterer15}
Ketterer C., Obata's rigidity theorem for metric measure spaces, \href{https://doi.org/10.1515/agms-2015-0016}{\textit{Anal.
 Geom. Metr. Spaces}} \textbf{3} (2015), 278--295, \href{https://arxiv.org/abs/1410.5210}{arXiv:1410.5210}.

\bibitem{kettererHK}
Ketterer C., The {H}eintze--{K}archer inequality for metric measure spaces,
 \href{https://doi.org/10.1090/proc/15041}{\textit{Proc. Amer. Math. Soc.}} \textbf{148} (2020), 4041--4056,
 \href{https://arxiv.org/abs/1908.06146}{arXiv:1908.06146}.

\bibitem{kila}
Kitabeppu Y., Lakzian S., Characterization of low dimensional {${\rm
 RCD}^*(K,N)$} spaces, \href{https://doi.org/10.1515/agms-2016-0007}{\textit{Anal. Geom. Metr. Spaces}} \textbf{4} (2016),
 187--215, \href{https://arxiv.org/abs/1505.00420}{arXiv:1505.00420}.

\bibitem{Klartag17}
Klartag B., Needle decompositions in {R}iemannian geometry, \href{https://doi.org/10.1090/memo/1180}{\textit{Mem. Amer.
 Math. Soc.}} \textbf{249} (2017), v+77~pages, \href{https://arxiv.org/abs/1408.6322}{arXiv:1408.6322}.

\bibitem{KunzingerSteinbauerStojkovicVickers15}
Kunzinger M., Steinbauer R., Stojkovi\'c M., Vickers J.A., Hawking's
 singularity theorem for {$C^{1,1}$}-metrics, \href{https://doi.org/10.1088/0264-9381/32/7/075012}{\textit{Classical Quantum
 Gravity}} \textbf{32} (2015), 075012, 19~pages, \href{https://arxiv.org/abs/1411.4689}{arXiv:1411.4689}.

\bibitem{liwei_bakry_emery}
Li H., Wei Y., {$f$}-minimal surface and manifold with positive
 {$m$}-{B}akry--\'Emery {R}icci curvature, \href{https://doi.org/10.1007/s12220-013-9434-5}{\textit{J.~Geom. Anal.}} \textbf{25}
 (2015), 421--435, \href{https://arxiv.org/abs/1209.0895}{arXiv:1209.0895}.

\bibitem{liwei_rigidity}
Li H., Wei Y., Rigidity theorems for diameter estimates of compact manifold
 with boundary, \href{https://doi.org/10.1093/imrn/rnu052}{\textit{Int. Math. Res. Not.}} \textbf{2015} (2015),
 3651--3668, \href{https://arxiv.org/abs/1306.3715}{arXiv:1306.3715}.

\bibitem{martinli}
Li M.M.-C., A sharp comparison theorem for compact manifolds with mean convex
 boundary, \href{https://doi.org/10.1007/s12220-012-9381-6}{\textit{J.~Geom. Anal.}} \textbf{24} (2014), 1490--1496,
 \href{https://arxiv.org/abs/1204.1695}{arXiv:1204.1695}.

\bibitem{lottvillani}
Lott J., Villani C., Ricci curvature for metric-measure spaces via optimal
 transport, \href{https://doi.org/10.4007/annals.2009.169.903}{\textit{Ann. of Math.}} \textbf{169} (2009), 903--991,
 \href{https://arxiv.org/abs/math.DG/0412127}{arXiv:math.DG/0412127}.

\bibitem{LuMinguzziOhta19+}
Lu Y., Minguzzi E., Ohta S.-I., Geometry of weighted {L}orentz--{F}insler
 manifolds I: Singularity theorems, \href{https://arxiv.org/abs/1908.03832}{arXiv:1908.03832}.

\bibitem{McCann18+}
McCann R.J., Displacement convexity of {B}oltzmann's entropy characterizes the
 strong energy condition from general relativity, \href{https://dx.doi.org/10.4310/CJM.2020.v8.n3.a4}{\textit{Camb.~J. Math.}}
 \textbf{8} (2020), 609--681, \href{https://arxiv.org/abs/1808.01536}{arXiv:1808.01536}.

\bibitem{minguzzi}
Minguzzi E., Lorentzian causality theory, \href{https://doi.org/10.1007/s41114-019-0019-x}{\textit{Living Rev. Relativity}}
 \textbf{22} (2019), 3, 202~pages.

\bibitem{NakajimaShioya19}
Nakajima H., Shioya T., Isoperimetric rigidity and distributions of
 1-{L}ipschitz functions, \href{https://doi.org/10.1016/j.aim.2019.04.043}{\textit{Adv. Math.}} \textbf{349} (2019), 1198--1233,
 \href{https://arxiv.org/abs/1801.01302}{arXiv:1801.01302}.

\bibitem{ohtmea}
Ohta S.-I., On the measure contraction property of metric measure spaces,
 \href{https://doi.org/10.4171/CMH/110}{\textit{Comment. Math. Helv.}} \textbf{82} (2007), 805--828.

\bibitem{ohtafinsler1}
Ohta S.-I., Finsler interpolation inequalities, \href{https://doi.org/10.1007/s00526-009-0227-4}{\textit{Calc. Var. Partial
 Differential Equations}} \textbf{36} (2009), 211--249.

\bibitem{sakurai}
Sakurai Y., Rigidity of manifolds with boundary under a lower {B}akry--\'Emery
 {R}icci curvature bound, \href{https://doi.org/10.2748/tmj/1552100443}{\textit{Tohoku Math.~J.}} \textbf{71} (2019),
 69--109, \href{https://arxiv.org/abs/1506.03223}{arXiv:1506.03223}.

\bibitem{stugeo2}
Sturm K.-T., On the geometry of metric measure spaces.~{II}, \href{https://doi.org/10.1007/s11511-006-0003-7}{\textit{Acta Math.}}
 \textbf{196} (2006), 133--177.

\bibitem{viltot}
Villani C., Optimal transport. Old and new, \textit{Grundlehren der Mathematischen
 Wissenschaften}, Vol.~338, \href{https://doi.org/10.1007/978-3-540-71050-9}{Springer-Verlag}, Berlin, 2009.

\bibitem{wong}
Wong J., An extension procedure for manifolds with boundary, \href{https://doi.org/10.2140/pjm.2008.235.173}{\textit{Pacific~J.
 Math.}} \textbf{235} (2008), 173--199.

\end{thebibliography}
\end{document}